\documentclass[11pt]{amsart}
\usepackage{amsmath}
\usepackage{amssymb}
\usepackage{amsthm}
\usepackage{mathrsfs}
\usepackage{comment}
\usepackage{hyperref}

\usepackage[all,cmtip]{xy}\usepackage{xcolor}
\usepackage{enumerate}
\usepackage{bm}
\usepackage{dsfont}
\usepackage{mathtools}
\usepackage[cal=euler]{mathalfa}
\usepackage[top=1in, bottom=1.25in, left=1.25in, right=1.25in]{geometry}
\usepackage{parskip}

\hypersetup{colorlinks=true,linkcolor=magenta,citecolor=blue}

\DeclareMathAlphabet{\mathpzc}{OT1}{pzc}{m}{it}

\theoremstyle{plain}
\newtheorem{theorem}{Theorem}[section]
\newtheorem*{theorem*}{Theorem}

\newtheorem{lemma}[theorem]{Lemma}
\newtheorem*{claim*}{Claim}
\newtheorem{proposition}[theorem]{Proposition}

\newtheorem{corollary}[theorem]{Corollary}

\theoremstyle{definition}
\newtheorem{definition}[theorem]{Definition}

\newtheorem{example}[theorem]{Example}

\newtheorem{remark}[theorem]{Remark}

\newcommand{\ignore}[1]{}

\begin{document}
\setlength{\parindent}{0pt}

\author{S. Pumpl\"un}

\email{susanne.pumpluen@nottingham.ac.uk}
\address{School of Mathematical Sciences\\
University of Nottingham\\ University Park\\ Nottingham NG7 2RD\\
United Kingdom }

\date{\today}

\keywords{Nonassociative algebras, crossed products, linear codes.}

\subjclass[2020]{Primary: 17A35 Secondary: 17A60, 17A99}

\title[Menichetti's nonassociative $G$-crossed product algebras]
{Menichetti's nonassociative $G$-crossed product algebras  and Menichetti codes}

\maketitle

%submitted to DCC 8.9.26

\begin{abstract}
We define the nonassociative Menichetti algebras which can be viewed as nonassociative crossed product algebras  and then use them as ambient algebras for new linear error-correcting codes. More precisely, we take their left principal ideals to define linear codes which are in canonical one-to-one correspondence with these ideals, imitating the approach taken when defining right skew polycyclic codes as left principal ideals of Petit algebras.
The approach is novel and will create large classes of new linear codes which are well behave because of their algebraic definition as ideals of an algebra. 
 With the right choice of algebra the codes display symmetric and cyclic properties which promise efficient decoding algorithms.
  \end{abstract}

\section{Introduction}

Let $F$ be a field. All algebras we will consider are unital. An associative algebra $A$ over $F$ is called a \emph{$G$-crossed product algebra} or
\emph{crossed product algebra}, if it contains a maximal field extension $K/F$ which is Galois
with Galois group $G = {\rm Gal}(K/F)$. Crossed product algebras play an important role in
the theory of central simple algebras, since every element in the Brauer
group of $F$ is similar to a crossed product algebra.

What is a reasonable definition of a nonassociative  $G$-crossed product algebra? The first paper addressing this question is by A. A. Albert \cite{A1}. The construction defined therein is very general and includes, for instance, even the iterated Cayley Dickson doubling process.

 In the first part of this paper, we define generalised Menichetti algebras. In 1975, Menichetti constructed a class of semifields of dimension $n^2$ over a finite field $F$ \cite{M}.
 We generalise Menichetti's approach  in Section \ref{sec:classical}.
Menichetti algebras are unital central nonassciative algebras of dimension $n^2$ that have a  Galois field extension $K/F$ of degree $n$ with Galois group $G$ in their nucleus.  The field $K$ is even  a maximal commutative subalgebra of the  nucleus of ${\rm Men}(K/F,k_0,\dots, k_{n-1})$. Thus any Menichetti algebra could be seen as nonassociative version of a crossed product algebra, assuming we use the above definition of an associative crossed product algebra that employs maximal subfields. Menichetti algebras form a class of nonassociative algebras of formidable ``size'', and depending on the choice of scalars used in their definition they can be made flexible or power associative, or even associative.

In the first generalisation, we define  nonassociative algebras ${\rm Men}(K/F,k_0,\dots, k_{n-1})$ of dimension $n^2$ over $F$, where
 $K/F$ is any Galois field extensions  with Galois group $G$; $G$ does not have to be cyclic.  We then construct nonassociative generalised Menichetti algebras ${\rm Men}(D,\sigma,k_0,\dots, k_{n-1})$ of dimension $n^2m^2$ over $F$ employing a central simple algebra $D$ of degree $m$ over some separable field extension $C$ of $F$, and some $\sigma\in {\rm Aut}(D)$ such that
$\sigma|_{C}$ has finite order $n$. Both our generalisations  make up a large class of non-trivial  semiassociative algebras with nucleus $K$, respectively $D$, and play an important role in the nonassociative Brauer monoid \cite{BHMRV}; in particular, the tensor product of a central simple associative algebra and a Menichetti algebra yields a generalised Menichetti algebra.

 Although we could argue that each Menichetti algebra can be seen as a nonassociative crossed product, we will use a different definition of a nonassociative $G$-crossed product algebra in this paper. 
 A prominent example of such a nonassociative crossed product algebra is the nonassociative cyclic algebra $(K/F,\sigma,a)$ where $K/F$ is a cyclic Galois field extension with Galois group $G$ generated by $\sigma$, and $a\in K\setminus F$ \cite{St, W}. This algebra canonically generalises the  cyclic central simple algebra $(K/F,\sigma,a)$, which is an associative crossed product algebra.  The principal left ideals of a nonassociative crossed product algebra $(K/F,\sigma,a)$ are in one-one correspondence with the well-known $(\sigma,a)$-skew
constacyclic codes of length $m=[K:F]$ over $K$. This can be used in classifications of these skew constacyclic codes, e.g. see \cite{NevPum2025, Pum2025}.

 We then look at possible code constructions in the second part of our paper. The overall push to study codes over more general finite rings naturally leads to the question what is the correct and most useful generalisation of linear error-correcting codes over rings which do not necessarily have a cyclic automorphism group? One answer is to use
different types of associative crossed products as ambient spaces, where $G$ need not be cyclic. These play an increasingly important role in
the theory of linear error-correcting codes, with crossed product codes being introduced as those two-sided ideals which are lattices of associative crossed products  by Ginosar and Moreno \cite{GM}. However, apart from the nonassociative cyclic algebras,
so far only associative crossed products have been employed, using twisted group algebras over finite fields, or skew group rings as the ambient spaces \cite{DW21, G, GM, SU25}. There is some overlap between Menichetti algebras and our $G$-crossed product algebras when $G$ is cyclic.

In our approach, we use the opposite Menichetti algebras as ambient spaces of linear codes, which means we define a code as in canonical one-on-one correspondence with a proper principal left ideal in its ambient algebra, that means we imitate the approach used in the definition of skew polycyclic codes as left principal ideals in their ambient Petit rings and look at left principal ideals only.

 We first determine when Menichetti algebras are division algebras in  Section  \ref{sec:results}. Since division algebras only have trivial two-sided ideals, these results give valuable information on which algebras to use in code constructions if we are only looking for codes associated with two-sided ideals, as in \cite{GM}. We then look at the special cases of Menichetti algebras of degree three and four  in Sections \ref{sec:three} and  \ref{sec:4} to obtain examples.
 We include  (more restrictively defined) nonassociative $G$-crossed product algebras of degree four to illustrate their different behaviour to  Menichetti algebras.

 When $[K:F]=2$, ${\rm Gal}(K/F)=\langle \sigma\rangle $, $k_0=1$ and $k_2=a$, Menichetti algebras are the opposite algebras of the nonassociative Petit algebras $K[t;\sigma]/K[t;\sigma](t^2-a)$.
 These algebras are exactly the nonassociative quaternion algebras $(K/F,\sigma,a)$, i.e. a well-known associative classical central simple $G$-crossed product algebras when $a\in F^\times$, and a proper nonassociative quaternion algebras when $a\in K\setminus F$ \cite{BP19}.
 This, however, is the ambient algebra of all  $(\sigma,a)$-skew constacyclic codes over $K$ of length $2$, which are in one-to-one correspondence with the  principal left ideals in $K[t;\sigma]/K[t;\sigma](t^2-a)$  \cite{Pum2017, Pum2025}.

This was observed more generally in \cite{GM} when $(K/F,\sigma,a)=K[t;\sigma]/K[t;\sigma](t^n-a)$ is an associative crossed product. The sometimes prevailing  claim that \emph{all} skew constacyclic codes are associative $G$-crossed product codes as defined in \cite{GM} is misleading, though, and not even true when we also consider the nonassociative crossed product ambient algebras $(K/F,\sigma,d)$ (i.e., the proper nonassociative cyclic algebras).

   We use the opposite algebras of  Menichetti algebras as the ambient algebras to define linear Menichetti codes that correspond to their principal left ideals in Section \ref{sec:codes}. The linear codes which correspond to principal left ideals have one generator, corresponding to the generator  of the left ideal and display certain symmetric or even cyclic properties when the Menichetti algebra is chosen wisely. These cyclic properties should yield codes with good minimum Hamming distances. Since they can be described by a single generator, this greatly reduces their storage and implementation complexity. This will lead to good linear codes that will be discussed in a forthcoming paper \cite{Pum25.2}.
    Most importantly, this approach will allow for code constructions that can use alphabets which do not just have a cyclic automorphism group, which will broaden the pool of available linear codes with cyclic properties over many finite rings.

We  define $(G,d)$-constacyclic  codes which are the obvious generalisations of skew constacyclic codes (Example \ref{ex:cyclic}) as an example of a Menichetti code which uses abelian noncyclic automorphism groups in their definition.
Since many finite rings do not have a cyclic automorphism group, our approach allows  more general code constructions, where the algebraic structure of the ambient algebra. The Hamming weight preserving isomorphisms of the ambient algebra can then be used later to investigate and classify the ensuing codes.

  Similar approaches are addressed in  \cite{DW21, DS,  G, GM, SU25}, but use only associative $G$-crossed product ambient spaces like twisted group algebras over finite fields or skew group rings.  For instance, all perfect linear codes are twisted group codes \cite{DW21}. It is now easy to see that the ambient algebra of a twisted group code is an associative Menichetti algebra.
  Then there are the skew $G$-codes,  a family of error-correcting codes defined as ideals in a skew group ring,
   where the ring is a finite commutative Frobenius ring and $G$ a finite group.
   Additionally,
   they encompass a large number of the known ``associative'' constructions in the literature for self-dual codes.

 Since Menichetti algebras are a ``very large'' class of algebras, restrictions to subclasses with  matrices that are good for code design is one of the challenges we need to address.
 Our approach mirrows the construction of skew polycyclic codes using nonassociative  ambient Petit algebras $S[t;\sigma]/S[t;\sigma]f$ \cite{Pum2017, Pum2025}, and polycyclic codes using the associative ambient algebras $S[t]/S[t]f$.  There are three setups that are of special interest.
\\\\
 1.
 $S_0=F$ is a finite field. Then all Galois extensions of $F$ will be cyclic. Nonetheless, we still obtain new linear  codes over $F$ as not all ambient Menichetti algebras are nonassociative cyclic algebras, which are the ambient algebras of skew constacyclic codes.
\\ 2. The code alphabet $S$ is a finite ring and ${\rm Aut}_{S_0}(S)$ is not cyclic.
\\ 3. We choose $\sigma=id$. Then we obtain Menichetti codes which are  generalisation of constacyclic codes which have  the associative quotient algebras $S[t]/S[t](t^n-a)$ as ambient algebras.
\\\\
A few very special cases of some of the results in Section  \ref{sec:results} appeared in the PhD thesis of A. Steele \cite{St}.

\section{The construction} \label{sec:classical}

\subsection{Preliminaries} Let $F$ be a field.  A \emph{nonassociative algebra} $A$ over $F$ is an $F$-vector space together with
 an $F$-bilinear map $A\times A\to A$, $(x,y) \mapsto x \cdot y$, denoted simply by juxtaposition $xy$, which is
the multiplication of $A$.
 $A$ is called \emph{unital} if there is an element in $A$, denoted by 1, such that $1x=x1=x$ for all $x\in A$.
We only consider finite-dimensional unital algebras and call such an algebra a \emph{proper nonassociative algebra} if it is not associative.
 An algebra $A\not=0$ is called a {\it division algebra} if for any $0\not=a\in A$,
the left and right multiplication  with $a$, $L_a(x)=ax$ and $R_a(x)=xa$, are bijective.
Since our $A$ is finite-dimensional as an $F$-vector space,  $A$ is a division algebra if and only if $A$ has no zero divisors.
 The {\it left, middle} and \emph{right nucleus} of $A$ are defined as ${\rm
Nuc}_l(A) = \{ x \in A \, \vert \, [x, A, A]  = 0 \}$,  ${\rm Nuc}_m(A) = \{ x \in A \, \vert \, [A, x, A]  = 0 \}$
${\rm Nuc}_r(A) = \{ x \in A \, \vert \, [A,A, x]  = 0 \}$, respectively, where $[x, y, z] =
(xy) z - x (yz)$ is the {\it associator}. These nuclei are associative
subalgebras of $A$. Their intersection
 ${\rm Nuc}(A) = \{ x \in A \, \vert \, [x, A, A] = [A, x, A] = [A,A, x] = 0 \}$ is called the {\it nucleus} of $A$.
The nucleus of $A$ is an associative subalgebra of $A$,
and $x(yz) = (xy) z$ whenever one of the elements $x, y, z$ is in
${\rm Nuc}(A)$. The \emph{commutator} is
 ${\rm Comm}(A)=\{x\in A\,|\, xy=yx \text{ for all } y\in A\}$ and the
 {\it center} of $A$ is ${\rm C}(A)=\{x\in \text{Nuc}(A)\cap {\rm Comm}(A)\}$.
 A nonassociative algebra is called a \emph{proper} nonassociative algebra if it is not associative.
 The algebra $A $ is called \emph{simple}, if $A\not =\{0\}$ and $A$ has no ideals apart from $\{0\}$ and $A$, and \emph{central} or \emph{$F$-central}, if $C(A)=F$.

 An algebra $A$ over $F$ is called an  \emph{\'etale algebra} if it is a finite direct product of finite separable field extensions of $F$.  An \'etale algebra is semisimple and the class of \'etale algebras is closed under tensor products.

A finite dimensional nonassociative $F$-central algebra $A$ is called \emph{semiassociative} if its nucleus has an \'etale $F$-subalgebra  $E$, such that $A$ is cyclic and faithful
over $E\otimes_F E$ via the action $(e\otimes e')a=eae'$ for all $a\in A$, $e,e'\in E$ \cite{BHMRV}.
The dimension of a semiassociative algebra $A$ is a square \cite[Corollary 3.4]{BHMRV} and the root of this dimension is called the
\emph{degree} of $A$. Every associative central simple algebra of degree $n$ has a maximal \'etale subalgebra $E$ of dimension $n$ and is semiassociative.

If $A$ is a
semiassociative algebra with respect to one \'etale subalgebra of its nucleus, then it is
 semiassociative
with respect to all \'etale subalgebras of its nucleus \cite[Proposition 3.6]{BHMRV}.
 If A is semiassociative of degree $n$, then any $n$-dimensional
\'etale subalgebra $E$ of ${\rm Nuc}(A)$ is a maximal commutative subalgebra of $A$ \cite[Corollary 7.3]{BHMRV}.
A semiassociative algebra $A$ over $F$ is called \emph{semicentral}, if all of the simple components of the semisimple quotient $\sigma(A)={\rm Nuc}(A)/J({\rm Nuc}(A))$ are $F$-central \cite[Definition 16.1]{BHMRV}.
 A semiassociative algebra is \emph{homogeneous} if it is
semicentral, and the atoms are all Brauer equivalent to each other \cite[Definition 17.2]{BHMRV}.

Two semiassociative algebras $A$ and $B$ over $F$ are  \emph{Brauer equivalent}, if there exist skew matrix algebras $M_n(F;c)$ and $M_m(F,c')$
such that $A\otimes_F M_n(F;c) \cong B \otimes_F M_m(F;c') $. The \emph{semiassociative Brauer monoid} $Br^{sa}(F)$ is the set of equivalence classes with respect to Brauer equivalence, with product $[A]^{sa}[B]^{sa}=[A\otimes_F B]^{sa}$ and unit element $[F]^{sa}$.

\subsection{Menichetti algebras}
We generalise the original definition of Menichetti algebras over finite fields given  in \cite{M} as follows.
Let $\mathbb{Z}_n$ be the ring of integers modulo $n$.  Let  $K/F$ be a Galois field extension  of degree $n$ with abelian Galois group, let us denote its elements by $\tau_i$ with $\tau_0=id$, so
$${\rm Gal}(K/F)=\{ \tau_0,\dots,\tau_{n-1} \}$$
i.e. in the following we choose to have $\tau_i$ represent one fixed element in $G$ (different choices may amount to different algebras, this remains to be investigated).

\begin{definition}
Let $k_i\in K^\times$, $i\in 0,\dots,{n-1}$, and put $c_{i,0}=c_{0,j}=1$ and
$$c_{i,j}=k_0^{-1}k_1^{-1} \cdots    k_{j-1}^{-1}   k_ik_{i+1}\cdots k_{i+j-1} \text{ for all nonzero } i,j\in \mathbb{Z}_n,$$
where we will read the indices modulo $n$.

Let us first observe a certain symmetry that arises from this choice of parameters.
\begin{lemma}\label{le:sym}
Define $l_0=1$, $l_m=k_0k_1\cdots k_{m-1}$ and $l=l_n$. Then
$$
  c_{i,j}=\frac{l_{(i+j)\bmod n}}{l_il_j}\,l^{\varepsilon},
$$
where $\varepsilon=1$ if $i+j\ge n$ and 0 otherwise.
Then we have
$$
  c_{i,j}\,c_{i+j,l}=c_{j,l}c_{i,j+l}=\tfrac{l_{i+j+l}}{l_il_jl_l}
  l^{\lfloor (i+j+l)/n\rfloor}.$$
\end{lemma}

Define a multiplication on the $K$–vector space
\[
A = \bigoplus_{i=0}^{n-1} K z_{i},
\]
 via the rules
$$(az_i)\cdot  (bz_j)=\tau_j(a)b (z_i\cdot  z_j),$$
$$z_i\cdot  z_0=z_0\cdot  z_i =z_i \text{ for all } i\in\mathbb{Z}_n,$$
$$z_i\cdot  z_j=c_{ji}z_{i+j} \text{ for all } i,j\in\mathbb{Z}_n\setminus \{0\}$$
for all $a,b\in K$, where again we read the indices modulo $n$.
The resulting  nonassociative unital algebra $A$ over $F$  is called a \emph{Menichetti algebra of degree $n$} and denoted by ${\rm Men}(K/F, k_0,\dots,k_{n-1})$.
\end{definition}

Note that $c_{i,j}=c_{j,i}$ by definition.

The algebra $A={\rm Men}(K/F, k_0,\dots,k_{n-1})$ is a nonassociative unital algebra over $F$ of dimension $n^2$.  It is easy to see that $K$ is a subfield of ${\rm Men}(K/F, k_0,\dots,k_{n-1})$.

\begin{lemma}
Define
\begin{equation}\label{equ:1}
 M(x_0,\dots,x_{n-1})=\left [\begin {array}{cccccc}
x_0 &  c_{{n-1},1}\tau_1(x_{n-1}) & ... & c_{1,{n-1}} \tau_{n-1}(x_{1})\\
x_1 & \tau_1(x_0) & ... &  c_{2,{n-1}} \tau_{n-1}(x_{2})\\
x_2 & c_{1,1}\tau_1(x_{1}) & \tau_2(x_0) & c_{3,{n-1}} \tau_{n-1}(x_{3})\\
...& ...  & ... & ...\\
x_{n-2} &  c_{n-3,1}\tau_1(x_{n-3})  & ...& c_{{n-1},{n-1}} \tau_{n-1}(x_{n-1})\\
x_{n-1} &  c_{n-2,1}\tau_1(x_{n-2})  & ...& \tau_{n-1}(x_{0})\\
\end {array}\right ]
 \end{equation}
and identify $x_0z_0+\dots +x_{n-1}z_{n-1}$ with $(x_0,\dots,x_{n-1})$. Then the multiplication of the Menichetti algebra
 ${\rm Men}(K/F, k_0,\dots,k_{n-1})$ can also be expressed as
 $$(x_0,\dots,x_{n-1})\cdot  (y_0,\dots,y_{n-1})=M(x_0,\dots,x_{n-1})(y_0,\dots,y_{n-1})^t.$$
\end{lemma}

 When $G=\langle \sigma \rangle$ is cyclic and $\tau_i=\sigma^i$ for all  $ i,j\in\mathbb{Z}_n$, then $M(x_0,\dots,x_{n-1})$ is the same matrix as given in \cite[(11)']{M}.

Menichetti algebras  are important examples of semiassociative algebras in the Brauer monoid $Br^{sa}(F)$.

\begin{proposition}\label{prop:nuc Galois}
 (i) The field $K$ is contained in the left, middle and right
 nucleus of \\
${\rm Men}(K/F,k_0,\dots, k_{n-1})$.
\\ (ii) ${\rm Men}(K/F,k_0,\dots, k_{n-1})$ is a semiassociative algebra and  $K$  is a maximal commutative subalgebra of ${\rm Men}(K/F,k_0,\dots, k_{n-1})$.
\\ (iii) Every proper nonassociative algebra ${\rm Men}(K/F,k_0,\dots, k_{n-1})$  is not homogeneous.
\end{proposition}

\begin{proof} $(i)$ The proof is straightforward.
\\ $(ii)$ The algebra ${\rm Men}(K/F,k_0,\dots, k_{n-1})$ is cyclic and faithful
over $K\otimes_F K$ via the action $(e\otimes e')a=eae'$ for all $a\in {\rm Men}(K/F,k_0,\dots, k_{n-1})$, $e,e'\in K$. Therefore,  ${\rm Men}(K/F,k_0,\dots, k_{n-1})$ is a semiassociative algebra.
 It follows from \cite[Corollary 7.3]{BHMRV} that $K$  is a maximal commutative subalgebra of ${\rm Men}(K/F,k_0,\dots, k_{n-1})$.
\\ $(iii)$  Since $K\subset {\rm Nuc}({\rm Men}(K/F,k_0,\dots, k_{n-1}))$ we know that ${\rm Men}(K/F,k_0,\dots, k_{n-1})$ is not semicentral.
Since ${\rm Men}(K/F,k_0,\dots, k_{n-1})$  is not semicentral, every proper nonassociative algebra ${\rm Men}(K/F,k_0,\dots, k_{n-1})$ is a nontrivial example of a semiassociative algebra: it does not lie in the similarity class of any $F$-central simple algebra $B$ in $Br^{sa}(F)$ and is thus not homogeneous \cite{BHMRV}.
 \end{proof}

  \begin{proposition}\label{prop:noncylic}
Suppose that $L/F$ is a cyclic Galois field extension of degree $n$ with ${\rm Gal}(L/F)=\langle \sigma\rangle$ and that $A=L[t;\sigma]/L[t;\sigma](t^n-a)$ is a proper nonassociative Petit algebra.
 Let $K/F$ be any Galois field extension of degree $n$ that is not cyclic. Then neither $A$ nor $A^{op}$ are isomorphic to a Menichetti algebra  ${\rm Men}(K/F, k_0,\dots,k_{n-1})$ for any $k_i\in K^\times$.
\end{proposition}

\begin{proof}
The degree $n$ field extension $K/F$ is contained in the middle nucleus of ${\rm Men}(K/F,k_0,\dots, k_{n-1})$. The middle nucleus of  $A$ is the Galois field extension $L/F$ of degree $n$. Suppose $L[t;\sigma]/L[t;\sigma](t^n-a)^{op}\cong
  (K/F, k_0,\dots,k_{n-1})$. Then ${\rm Nuc}_m(A)={\rm Nuc}_m({\rm Men}(K/F,k_0,\dots, k_{n-1}))$, and $K\subset L$, a contradiction.
\end{proof}

\subsection{Classical associative crossed product algebras}

Let $K/F$ be a Galois field extension with Galois group $G$. An associative ($G$-)\emph{crossed product algebra} $A=(K,G,{\bf c})=\bigoplus_{g\in G}Kz_g$ over $F$ can be defined as follows. Define a map (also called a \emph{cocycle}) ${\bf c}:G\times G\to K^\times$, $(g,h)\mapsto c_{g,h}$ such that
$$c_{g,id}=c_{id,g}=1_K$$
 for all $g\in G$ (so that  $A$ has a unit element), and such that
$$c_{g,h}c_{g\circ h,u}=g(c_{h,u})c_{g,h\circ u}$$
for all $g,h,u\in G$.
(This is the cocycle condition which is needed for associativity of the algebra we define now).
 Then elements $\sum_{g\in G} a_gz_g\in A$ are added term-wise and multiplied according to the rule
$$(az_g)(b z_h)=a g(b) c_{g,h} z_{g\circ h}$$
for all $a, b\in K$,
which is extended bilinearily. There is also a $K$-bimodule structure on $A$ given by $a(b z_g)= (ab)z_g$ and $(a z_g) b=a g(b) z_g$.
The opposite algebra of an associative $G$-crossed product algebra is again a $G$-crossed product algebra,
with multiplication now defined via
$$(az_g)(b z_h)=h(a) b c_{h,g} z_{g\circ h}$$
for all $a, b\in K$, $g,h\in G$.

\ignore{say $G$ is abelian, then
$z_gz_h= c_{g,h}z_{gh}= c_{g,h}z_{hg}$ and $z_h z_g=c_{h,g} z_{hg}$. If $c_{g,h}=c_{h,g}$ then we get
$$z_gz_h= c_{g,h}z_{gh}= c_{h,g}z_{hg}=z_h z_g.$$
}

\begin{definition}
Let $K/F $ be a Galois field extension with  Galois group $G$,
and let ${\bf c} : G \times G \to K^{\times}$, $(g,h)\mapsto c_{g,h}$ be a map such that $c_{id,h}=c_{g,id}=1$ for all $g,h\in G$.
Then the $K$–vector space
\[
A = \bigoplus_{g \in G} K z_{g},
\]
with the multiplication induced via
\[
(a z_g)(b z_h) = a g(b) {\bf c}(g,h) z_{g\circ h}= a g(b) c_{g,h}\, z_{g\circ h}\text{ for all }
 a,b \in K,\ g,h \in G.
\]
is a unital algebra denoted by  $A = (K, G,  {\bf c})$ and  called a  \emph{nonassociative ($G$-)crossed product algebra}.
\end{definition}

This algebra is associative if and only if
\[
c_{g,h}\,c_{g\circ h,u}
  = g(c_{h,u})\,c_{g,h\circ u}
  \quad \text{for all } g,h,u \in G.
\]
Its opposite algebra has  multiplication given by
$$(az_g)(b z_h)=h(a) b c_{h,g} z_{g\circ h}$$
for all $a, b\in K$, $g,h\in G$.

\begin{example}
When $G=\langle \sigma \rangle$ is cyclic and we choose $\tau_i=\sigma^i$,  then the multiplication in ${\rm Men}(K/F, k_0,\dots,k_{n-1})$ coincides with the one of the opposite algebra of the nonassociative  $G$-crossed product algebra $(K,G, {\bf c})$
choosing  ${\bf c}$  as ${\bf c}(\sigma^i, \sigma^j)=c_{i,j}$ for all nonzero $ i,j\in\mathbb{Z}_n$.
Then $z_i\cdot  z_j=c_{ji}z_{i+j}$ can be read as
$$z_{\sigma^i}\cdot  z_{\sigma^j}=c_{ji}z_{\sigma^{i+j}}=c_{ji}z_{\sigma^{i}\circ \sigma^{j}}.$$
\end{example}

 \begin{remark}
Nonassociative $G$-crossed product algebras are special cases of Albert's nonassociative crossed extensions \cite{A1}.  In Albert's terminology,
what we do here is  ``extend'' a Galois field extension $K/F$ with Galois group $G$ and  construct special cases of  his crossed extension of the kind $(K,G,{\bf g})$, where ${\bf g}$ is what Albert calls an extension set. To construct a crossed extension $(K,G,{\bf g})$, Albert only requires $G$ to be a finite subset of $Gl_F(A)$ where every $F$-linear bijective map $\sigma:K\to K$ in $G$  satisfies $\sigma(1)=1$,
and ${\bf g}=\{g_{i,j}\}$ to be any set of elements in $K^\times$ for  $i,j\in \mathbb{Z}_n$. For a field extension $K/F$ of dimension $n$, Albert also defines a direct generalization of the Petit algebra $K[t;\sigma]/K[t;\sigma](t^n-d)$, these algebras were re-introduced much later in \cite{P66}.
 \end{remark}

 Let us make the following trivial but important observations about our definition of nonassociative crossed product algebras. Over fields, we made our definition general enough to include all important cases for code constructions over finite fields, yet not as general as possible (e.g. using a nonassociative generalisation of the very general associative setup  in \cite{OeLun}).

 Secondly, we defined nonassociative crossed product algebras $(K/F,G, {\bf c})$ using Galois field extensions $K/F$ of degree $n$ and their Galois groups $G={\rm Gal}(K/F)$ and refrained from defining more general looking $H$-crossed product algebras $(K/F,H, {\bf c})$ where we instead choose a subgroup $H$ of $G$ of order $m$ in the definition (abusing notation here, as the definition would be analogously as the one given above). This is because a straightforward check reveals that  any $H$-crossed product algebra $(K/F,H, {\bf c})$ is not just an algebra over $F$, but also an algebra over the fixed field $E={\rm Fix}(H)$ of $H$. Moreover, by Galois Theory, and viewed as $E$-algebra, the $H$-crossed product algebra $(K/F,H, {\bf c})$ is canonically isomorphic to the $E$-algebra $(K/E,H, {\bf c})$. In other words, for any subgroup $H$ of $G$, a
 $H$-crossed product algebra $(K/F,H, {\bf c})$ which has  dimension $[K:F]\cdot |H|=mn$ over $F$ is isomorphic to the algebra $(K/E,{\rm Gal}(K/E), {\bf c})=(K/E,H, {\bf c})$, viewed as an algebra over $F$.

\subsection{Generalized Menichetti algebras}
 Let $D$ be a central simple algebra of degree $m$  over its center $C$. Let $\tau_1,\dots,\tau_{n-1}\in {\rm Aut}(D)$ such that $\{\tau_0, \tau_1,\dots,\tau_{n-1}\}$ is an abelian subgroup of ${\rm Aut}(D)$ with $\tau_0=id$, and such that
${\tau_i}|_{C}$ has finite order $n_i$.

Put $F_i={\rm Fix}(\tau_i)\cap C$. Then
  $C/F_i$ is a cyclic Galois field extension of degree $n_i$ with $\mathrm{Gal}(C/F_i) = \langle {\tau_i}|_{C} \rangle$. Define $F=F_1\cap\dots\cap F_n$. Suppose that
   $C/F$ is an abelian Galois field extension with Galois group $G=\{{\tau_1}|_{C},\dots ,{\tau_{n-1}}|_{C}\}$.

\begin{definition}
Let $k_i\in C^\times$, $i\in 0,\dots,{n-1}$, and put
$$c_{i,j}=k_0^{-1}k_1^{-1} \cdots    k_{j-1}^{-1}   k_ik_{i+1}\cdots k_{i+j-1}$$
for all nonzero $i,j\in \mathbb{Z}_n$ and $c_{i,0}=c_{0,j}=1$. For all $x_i\in D$, define
\[ M(x_0,\dots,x_{n-1})=\left [\begin {array}{cccccc}
x_0 &  c_{{n-1},1}\tau_1(x_{n-1}) & ... & c_{1,{n-1}} \tau_{n-1}(x_{1})\\
x_1 & \tau_1(x_0) & ... &  c_{2,{n-1}} \tau_{n-1}(x_{2})\\
x_2 & c_{1,1}\tau_1(x_{1}) & \tau_2(x_0) & c_{3,{n-1}} \tau_{n-1}(x_{3})\\
...& ...  & ... & ...\\
x_{n-2} &  c_{n-3,1}\tau_1(x_{n-3})  & ...& c_{{n-1},{n-1}} \tau_{n-1}(x_{n-1})\\
x_{n-1} &  c_{n-2,1}\tau_1(x_{n-2})  & ...& \tau_{n-1}(x_{0})\\
\end {array}\right ].\]
 Define a multiplication on the left $D$-module $A=D^n$ via
 $$(x_0,\dots,x_{n-1})\cdot  (y_0,\dots,y_{n-1})=M(x_0,\dots,x_{n-1})(y_0,\dots,y_{n-1})^t.$$
 We obtain a \emph{generalized Menichetti algebra of degree $mn$} we denote by ${\rm Men}(D, G, k_0,\dots,k_{n-1})$.
\end{definition}

The algebra ${\rm Men}(D, G, k_0,\dots,k_{n-1})$ is unital, has dimension $n^2m^2$ over $F$ and  $D$ as a subalgebra.

\begin{example}
 Let $D_0$ be a central simple algebra over $F$ of degree $m$.
Let  $K/F$ be a cyclic field extension of degree $n$  with ${\rm Gal}(K/F)=\langle \sigma\rangle$.
 We look at the special case that $D=D_0\otimes_F K$. Let $\widetilde{\sigma}$ be the unique extension of $\sigma$ to $D$ such that $\widetilde{\sigma}|_{D_0}=id_{D_0}$.
   Then the tensor product
$$D_0\otimes_F {\rm Men}(K/F, k_0,\dots,k_{n-1})\cong {\rm Men}(D, \langle\widetilde{\sigma} \rangle, k_0,\dots,k_{n-1})$$
is a nonassociative generalised Menichetti algebra
over $F$ of degree $mn$.
   In the semiassociative Brauer monoid $Br(F)^{sa}$, we thus obtain
$$[D_0]^{sa} [{\rm Men}(K/F, k_0,\dots,k_{n-1})]^{sa} =[ {\rm Men}(D,\langle \widetilde{\sigma}\rangle, k_0,\dots,k_{n-1})]^{sa}.$$
\end{example}

Let now $K/F$ be a Galois field extension of degree $n$  with $G={\rm Gal}(K/F)=\{ \tau_1,\dots,\tau_m\}$. Let $\widetilde{\tau_i}$ be the unique extension of $\tau_i$ to $D$ such that $\widetilde{\tau_i}|_{D_0}=id_{D_0}$.  Then $\widetilde{G}=\{\widetilde{\tau}_1,\dots,\widetilde{\tau}_{n-1}\}$ is a subgroup of ${\rm Aut}(D)$. The tensor product of two semiassociative algebras is semiassociative, thus
$$A=D_0\otimes_F {\rm Men}(K/F,G,  k_0,\dots,k_{n-1})$$
is the semiassociative algebra ${\rm Men}(D, \widetilde{G}, k_0,\dots,k_{n-1})$ over $F$ of degree $mn$

Since ${\rm Nuc}_m(A)=D_0\otimes {\rm Nuc}_m({\rm Men}(K/F, k_0,\dots,k_{n-1}))$, ${\rm Nuc}_r(A)=D_0\otimes {\rm Nuc}_r({\rm Men}(K/F, k_0,\dots,k_{n-1}))$ and ${\rm Nuc}_l(A)=D_0\otimes {\rm Nuc}_l({\rm Men}(K/F, k_0,\dots,k_{n-1}))$, we have
$D\subset {\rm Nuc}_m(A)$, $D\subset {\rm Nuc}_r(A)$ and $D\subset {\rm Nuc}_l(A)$.
 Thus $A$ is not semicentral and hence it does not lie in the similarity class of any $F$-central simple algebra $B$ in $Br^{sa}(F)$ as then $A$ is not homogeneous \cite{BHMRV}.

\begin{example}
In particular, let $\sigma\in {\rm Aut}(D)$ such that
$\sigma|_{C}$ has finite order $n$, and put $F={\rm Fix}(\sigma)\cap C$. Then
  $C/F$ is a cyclic Galois field extension of degree $n$ with $G=\mathrm{Gal}(C/F) = \langle \sigma|_{C} \rangle$.
For  all $x_i\in D$, then
\[ M(x_0,\dots,x_{n-1})=\left [\begin {array}{cccccc}
x_0 &  c_{{n-1},1}\sigma(x_{n-1}) & ... & c_{1,{n-1}} \sigma^{n-1}(x_{1})\\
x_1 & \sigma(x_0) & ... &  c_{2,{n-1}} \sigma^{n-1}(x_{2})\\
x_2 & c_{1,1}\sigma(x_{1}) & \sigma^2(x_0) & c_{3,{n-1}} \sigma^{n-1}(x_{3})\\
...& ...  & ... & ...\\
x_{n-2} &  c_{n-3,1}\sigma(x_{n-3})  & ...& c_{{n-1},{n-1}} \sigma^{n-1}(x_{n-1})\\
x_{n-1} &  c_{n-2,1}\sigma(x_{n-2})  & ...& \sigma^{n-1}(x_{0})\\
\end {array}\right ]
 \]
  defines the left multiplication of ${\rm Men}(D, G, k_0,\dots,k_{n-1})$.
In this case, we also denote the algebra also by ${\rm Men}(D, \sigma, k_0,\dots,k_{n-1})$.
\end{example}

\begin{lemma}
Let $E$ be a maximal \'etale subalgebra of $D$ of dimension $m$.  Then
  $E$ lies in the  left, middle and right nucleus of ${\rm Men}(D, \sigma, k_0,\dots,k_{n-1})$. Therefore ${\rm Men}(D, G, \sigma, k_0,\dots,k_{n-1})$ is a semiassociative algebra that is not homogeneous.
\end{lemma}
\begin{proof}
  We know that
  ${\rm Men}(D, \sigma, k_0,\dots,k_{n-1})$ has the \'etale algebra $E/F$ of dimension $mn$ in its nucleus. The algebra  ${\rm Men}(D, \sigma, k_0,\dots,k_{n-1})$ is a faithful $E^e$-module, thus is cyclic
  as a $E^e$-module \cite[Remark 3.3]{BHMRV} and is a semiassociative algebra \cite{BHMRV}.
  This implies that  ${\rm Men}(D, \sigma, k_0,\dots,k_{n-1})$ is not semicentral.
Since ${\rm Men}(D, \sigma, k_0,\dots,k_{n-1})$  is not semicentral, it does not lie in the similarity class of any $F$-central simple algebra $B$ in $Br^{sa}(F)$ and is thus not homogeneous \cite{BHMRV}.
\end{proof}

%%%%%%%%%%%%%%%%%%%%%%%%%%%%%%%%%%%%%%%%%%%%%%%%%%%%%%%%%%%%%%%%%%%%%%%%%%%%%%%%%%%%%%%%%%%%%%%%%%%
%
\section{When a Menichetti algebra is a division algebra}\label{sec:results}
%
%%%%%%%%%%%%%%%%%%%%%%%%%%%%%%%%%%%%%%%%%%%%%%%%%%%%%%%%%%%%%%%%%%%%%%%%%%%%%%%%%%%%%%%%%%%%%%%%%%

 Let  $K/F$ be an abelian Galois field extension  of degree $n$ and as before, write ${\rm Gal}(K/F)=\{ \tau_0,\dots,\tau_{n-1} \}$ with $\tau_0=id$.
Let $k_i\in K^\times$, $i\in 0,\dots,{n-1}$, and
$$c_{i,j}=k_0^{-1}k_1^{-1} \cdots    k_{j-1}^{-1}   k_ik_{i+1}\cdots k_{i+j-1} \text{ for all nonzero } i,j\in \mathbb{Z}_n.$$

 The analogous arguments as the one used in \cite[Proposition 1]{M} show that
 $$\det (M(x_0,\dots,x_{n-1}))=\det (M(\tau_i(x_0),\dots,\tau_i(x_{n-1})))$$
  for all $x_i\in K$, $i\in\mathbb{Z}_n$. More precisely,
$\det (M(x_0, \dots,x_{n-1}))$ is  a sum of terms of the type
$$k_0^{n_0}k_1^{n_1} \cdots k_{n-1}^{n_{n-1}} \rho_0(x_0)\rho_1(x_1)\cdots \rho_{n-1}(x_{n-1})$$
with $n_i\in\mathbb{Z}$ and the $\rho_i$ are either 0 or $\rho_i(x_i)=\tau_{r_1} (x_i)\cdots \tau_{r_s}(x_i) $, where the $r_j\in\mathbb{Z}_n$ are all distinct. Therefore
$\det (M(x_0, \dots,x_{n-1}))$ is  a sum of elements that look like
$$k_0^{n_0} k_1^{n_1} \cdots k_{n-1}^{n_{n-1}} f_{n_0, n_1,\dots, n_{n-1}}(x_0,\dots,x_{n-1}),$$
where the $f_{n_0, n_1,\dots, n_{n-1}}(x_0,\dots,x_{n-1})$ are polynomial functions in $x_0,\dots,x_{n-1}$, $n_i\in \mathbb{Z}$, and
$f_{n_0, n_1,\dots, n_{n-1}}(x_0,\dots,x_{n-1})\in F$ for all choices of $x_i\in K$.
This is proved by using clever matrix manipulations completely analogously as in
 \cite[Corollary 1]{M}. For the explicit computation when $n=3$, see Section \ref{sec:three}, when $n=4$, see Section \ref{sec:4}. It immediately yields the following general result.

 \begin{theorem}
If all the $n$ elements
that appear in the determinant  in the sum of its terms are linearly independent over $F$, then ${\rm Men}(K/F, k_0,\dots,k_{n-1})$ is a division algebra over $F$.
 \end{theorem}

Since this is a very general criterium and not easily tractable,  we will focus on some special cases.
Consider ${\rm Men}(K/F,1,k,\dots,k,kk')$, i.e. pick $k,k'\in K^\times$ and set
$$k_0=1, k_1=k_2=\dots =k_{n-2}=k \text{ and } k_{n-1}=kk'.$$
 Then
 $$c_{i,j}=k \text{ for } i+j<n, c_{i,j}=kk'  \text{ for } i+j=n,  \text{ and } c_{i,j}=k' \text{ for } i+j>n,$$
 and
 \[ M(x_0,\dots,x_{n-1})=\left [\begin {array}{ccccccc}
x_0 &  kk'\tau_0(x_{n-1}) & kk' \tau_2(x_{n-2}) &... & kk' \tau_{n-1}(x_{1})\\
x_1 & \tau_0(x_0) & k' \tau_2(x_{n-1})&... &  k' \tau_{n-1}(x_{2})\\
x_2 & k\tau_1(x_{1}) & \tau_2(x_0) & ... & k' \tau_{n-1}(x_{3})\\
...& ...  & ... & ...\\
x_{n-2} &  k\tau_0(x_{n-3})  & ...& ...& k' \tau_{n-1}(x_{n-1})\\
x_{n-1} &  k\tau_0(x_{n-2}) &   k \tau_2(x_{n-3})  & ...& \tau_{n-1}(x_{0})\\
\end {array}\right ].
 \]

If $k=\alpha_1+\alpha_2k'$ with $\alpha_i\in F$, then $\det(M(x_0,\dots,x_{n-1}))$ is a polynomial of the form
$$\det M(x_0,\dots,x_{n-1})=N_{K/F}(x_0)+kk'f_1(x_0,\dots,x_{n-1})+\dots +kk'^{n-1}f_{n-1}(x_0,\dots,x_{n-1})$$
and $f_j(x_0,\dots,x_{n-1})\in F$ for all choices of $x_i\in K$. In particular,
$$f_0(x_0,\dots,x_{n-1})=x_0\tau_1(x_0)\cdots\tau_{n-1}(x_0)=N_{K/F}(x_0),$$
$$f_1(0,\dots,0,x_i,x_{i+1},\dots,x_{n-1})=(-1)^{(i+2)(m-i)}\alpha_1^{m-i-1}x_i \tau_0(x_i)\cdots \tau_{n-1}(x_i),$$
and consequently we find that
$$\det M(0,\dots,0,x_i,x_{i+1},\dots,x_{n-1})=kk'^i(f_1(0,\dots,0,x_i,x_{i+1},\dots,x_{n-1})+k'(\dots))$$
analogously as in \cite{M}.

\begin{theorem}\label{thm:main1}
If $k,k'\in K^\times$, $k=\alpha_1+\alpha_2k'$ with $\alpha_i\in F$, $\alpha_1\not=0$, are chosen such that $1,kk',kk'^2,\dots,kk'^{n-1}$ are linearly independent over $F$, then
${\rm Men}(K/F,1,k,\dots,k,kk')$ is a division algebra over $F$.
 \end{theorem}

 This generalises \cite[p. 344]{M}.

 \begin{proof}
 If $1,kk',kk'^2,\dots,kk'^{n-1}$ are linearly independent over $F$, then
 $$\det M(x_0,\dots,x_{n-1})=N_{K/F}(x_0)+kk'f_1(x_0,\dots,x_{n-1})+\dots +kk'^{n-1}f_{n-1}(x_0,\dots,x_{n-1})=0$$
  if and only if
 $f_0=\dots =f_{n-1}=0$, which in turn is equivalent to $x_0=\dots=x_{n-1}=0$, employing that
 $$\det M(0,\dots,0,x_i,x_{i+1},\dots,x_{n-1})=kk'^i(f_1(0,\dots,0,x_i,x_{i+1},\dots,x_{n-1})+k'(\dots)).$$
 \end{proof}

 \begin{corollary}  \label{cor:M1}
 (i) Choose $k',k=\alpha_1+\alpha_2k'\in K^\times$, $\alpha_i\in F$, $\alpha_1\not=0$, such that $1,k',k'^2,\dots,k'^{n-1}$ are linearly independent over $F$.
  Then
 $$k'^{m}=\sum_{i+0}^{n-1}\lambda_i k'^i$$
  for suitable $\lambda_i\in F$. If
 $$\alpha_1^{n-1}+\sum_{i=1}^{n-1} (-1)^{i+1}\lambda_{m-i}\alpha_1^{m-i-1} \alpha_2^{i}\not=0,$$
 then ${\rm Men}(K/F,1,k,\dots,k,kk')$ is a division algebra  over $F$.
 \\ (ii) Let $k,k'\in K^\times$.
 Choose
 $$k_0=1, k_1=\dots=k_{n-2}=k, k_{n-1}=kk', k'=\alpha_1+\alpha_2k\in K^\times, \alpha_i\in F, \alpha_1\not=0,$$
  such that $1,k'k,k'k^2,\dots,k'k^{n-1}$ are linearly independent over $F$. Then ${\rm Men}(K/F,1,k,\dots,k,kk')$ is a division algebra  over $F$.
 \end{corollary}

This generalises \cite[p. 345]{M}.

 \begin{proof}
 $(i)$ The elements $1$, $kk'=\alpha_1k'+\alpha_2k'^2, \dots,kk'^{n-2}=\alpha_1k'^{n-2}+\alpha_2k'^{n-1}$ and
 $kk'^{n-1}=\alpha_2\lambda_0+   \alpha_2\lambda_1k'+\dots+ (\alpha_1+\alpha_2 \lambda_{n-1})k'^{n-1}$  are linearly independent over $F$ if and only if
 \[\det(\left [\begin {array}{ccccccc}
1 & 0 & 0 &... & 0\\
0 & \alpha_1 & \alpha_1&... &  0\\
0& 0 & \alpha_1 & ... & 0\\
...& ...  & ... & ...\\
0 &  0  & ...& \alpha_1 & \alpha_2\\
\alpha_2 \lambda_0 &  \alpha_2 \lambda_1 & ...  & \alpha_2 \lambda_{n-2}& \alpha_1 +\alpha_2 \lambda_{n-1}\\
\end {array}\right ])\not=0
 \]
 for $(x_0,\dots,x_{n-1})\not=(0,\dots,0)$.
 \\ $(ii)$ The proof is analogous to the one for $(i)$:
 $$\det M(x_0,\dots,x_{n-1})=f_0(x_0,\dots,x_{n-1})+k' \sum_{j=1}^{n-1} k^jf_j(x_0,\dots,x_{n-1})$$
 and
 $f_0(x_0,\dots,x_{n-1})=N_{K/F}(x_0),$
 $$f_j(0,x_1,\dots,x_{m-j},0,\dots,0)=(-1)^{(j+2)(n-j)}\alpha_1^{n-j-1}N_{K/F}(x_{n-j}).$$
 \end{proof}

Menichetti algebras of degree 2 are well-known and are always crossed product algebras.

\begin{proposition}
(i) Every four-dimensional unital proper nonassociative  algebra over  $F$ which has a quadratic separable field extension of $F$ as its nucleus is the opposite algebra of  a
proper nonassociative Menichetti division algebra.
\\
(ii) The opposite algebras of  proper nonassociative Menichetti division algebras of dimension four over $F$ are exactly the nonassociative quaternion algebras.
\end{proposition}

\begin{proof}
Every four-dimensional unital proper nonassociative  algebra over  $F$ which has a quadratic separable field extension of $F$ as its nucleus is a nonassociative quaternion algebra \cite{W} and hence the opposite algebra of a Menichetti algebra. Thus the well-known nonassociative quaternion algebras are the only Menichetti division algebras of dimension four, up to isomorphism.
\end{proof}

In the next two sections we look at Menichetti algebras of degree three and four.

\section{Cubic Field Extensions} \label{sec:three}

Let $K/F$ be a cubic Galois field extension with Galois group $G=\langle \sigma \rangle$, norm $N_{K/F}$ and trace $T_{K/F}$. Let $a,b,c \in K^\times$ and put $\tau_i=\sigma^i$.
Then the Menichetti algebra ${\rm Men}(K/F, a,b,c) $ is a nonassociative $G$-crossed product algebra with the multiplication
\begin{align*}
x\cdot y &= (x_0y_0 +ca^{-1}\sigma(x_2) y_1  + ca^{-1}\sigma^2(x_1) y_2)z_0\\
&+ (x_1y_0 + \sigma(x_0) y_1  + cb^{-1} \sigma^2(x_2) y_2)z_1  \\
&+ (x_2y_0 + ba^{-1}\sigma(x_1) y_1 +\sigma^2(x_0)y_2 )z_2,
\end{align*}
 that is
$$
M(x_0, x_1, x_2)=
\left( \begin{array}{ccc}
x_0 &  ca^{-1}\sigma(x_2) &   ca^{-1}\sigma^2(x_1) \\
x_1     & \sigma(x_0) &   cb^{-1}\sigma^2(x_2)  \\
x_2    & ba^{-1}\sigma(x_1)& \sigma^2(x_0)
\end{array} \right).
$$
Moreover,
\[\det (M(y_0, y_1, y_2))= N_{K/F}(y_0) + ca^{-1}N_{K/F}(y_1) + b^{-1}c^2 a^{-1}N_{K/F}(y_2) - ca^{-1}T_{K/F}(y_0\sigma(y_1)\sigma^2(y_2)),\]
\[  =N_{K/F}(y_0) + b^{-1}c^2 a^{-1}N_{K/F}(y_2)+ ca^{-1}(N_{K/F}(y_1)-T_{K/F}(y_0\sigma(y_1)\sigma^2(y_2))) .\]
In particular, we have
\[z\cdot z = ba^{-1}z^2, \quad
 z^2\cdot z = z \cdot z^2 = ca^{-1},  \quad
 z^2 \cdot z^2 = cb^{-1}z  \quad
(z\cdot z)\cdot z =  ca^{-1}\sigma(ba^{-1}),  \quad
z\cdot (z\cdot z) =   ca^{-1}ba^{-1}.\]
This yields to the following observation.

\begin{lemma}
 If $ba^{-1}\in K\setminus F$, then ${\rm Men}(K/F,a,b,c)$ is not third power-associative.
\end{lemma}

\begin{example} (\cite{ST})
Let $d \in K\setminus F$.
For  $(K/F, d,1,d)$ we have
$$\det(M(x_0,x_1,x_2)) =N_{K/F}(x_0) + d N_{K/F}(x_2)+ (N_{K/F}(x_1)-T_{K/F}(x_0\sigma(x_1)\sigma^2(x_2))).$$
Thus if $y_2\not=0$ then $\det(M(y_0,y_1,y_2)) \not=0$.
This implies that if $x\cdot y=0$ and $y_2\not=0 $ then $x=0$. So there are no nontrivial right zero divisors $y\in {\rm Men}(K/F, d,1,d)$ with $y_2\not=0 $.
Similarly, if $x\cdot y=0$ and $y_2=0 $ but $N_{K/F}(x_0) \not=- N_{K/F}(x_1)$ then again $x\cdot y=0$ implies that $x=0$, so there are no nontrivial zero divisors
$x$, $y$ of this type.
\end{example}

 Let $k,k'\in K^\times$. The multiplication in ${\rm Men}(K/F,1,k,kk')$ is given by
\[M(x_0, x_1, x_2, x_3)=
\left( \begin{array}{ccc}
x_0 &  kk'\sigma(x_2) &   kk'\sigma^2(x_1) \\
x_1     & \sigma(x_0) &   k'\sigma^2(x_2)  \\
x_2    & k\sigma(x_1)& \sigma^2(x_0)
\end{array} \right) \]
Put $k=a^{-1}b$, $k'=b^{-1}c$, then
$(K/F,1,k,kk')=(K/F,a,b,c)$. We therefore investigate w.l.o.g.  algebras of the type ${\rm Men}(K/F,1,k,kk')$.

If $k,k'\in K^\times$, $k=\alpha_1+\alpha_2k'$ with $\alpha_i\in F$, $\alpha_1\not=0$, are chosen such that $1,kk',kk'^2$ are linearly independent over $F$, then
${\rm Men}(K/F,1,k,kk')$ is a division algebra over $F$ (Theorem \ref{thm:main1}). For $m=3$ we can prove this also directly.

\begin{proposition}\label{prop:div}
 If  1, $kk'=ca^{-1}$, $kk'^2=b^{-1}c^2 a^{-1}$ are linearly independent over $F$,
then ${\rm Men}(K/F, a,b,c)$ is a division algebra.
\end{proposition}
\begin{proof}
 Suppose that 1, $ca^{-1}\in K\setminus F$, and $b^{-1}c^2 a^{-1}\in K \setminus F$ are linearly independent over $F$.
Then
$N_{K/F}(x_0) + b^{-1}c^2 a^{-1}N_{K/F}(x_2)+ ca^{-1}(N(x_1)-T_{K/F}(x_0\sigma(x_1)\sigma^2(x_2))) =0$
implies that
$N_{K/F}(x_0)=0$, hence $x_0=0$, $N_{K/F}(x_2)=0$, hence $x_2=0$, and $N_{K/F}(x_1)-T_{K/F}(x_0\sigma(x_1)\sigma^2(x_2))=0$, i.e.
$N_{K/F}(x_1)=0$, that means $(x_0,x_1,x_2)=0$.
\end{proof}

The next result we obtain from Corollary \ref{cor:M1}.

 \begin{corollary} \label{cor:M3}
(i) Choose $k',k=\alpha_1+\alpha_2k'\in K^\times$, $\alpha_i\in F$, $\alpha_1\not=0$, such that $1,k',k'^2$ are linearly independent over $F$.
 Then $k'^{3}=\sum_{i=0}^{2}\lambda_i k'^i$ for suitable $\lambda_i\in F$. If
 $$\alpha_1^{2}+\sum_{i=1}^{2} (-1)^{i+1}\lambda_{3-i}\alpha_1^{3-i-1} \alpha_2^{i}\not=0,$$
 then ${\rm Men}(K/F,1,k,kk')$ is a division algebra  over $F$.
 \\ (ii) Choose  $k$, $k'=\alpha_1+\alpha_2k\in K^\times$, $\alpha_i\in F$, $\alpha_1\not=0$, such that $1,k'k,k'k^2$ are linearly independent over $F$. Then ${\rm Men}(K/F,1,k,kk')$ is a division algebra  over $F$.
 \end{corollary}

We now derive additional criteria for when ${\rm Men}(K/F, a,b,c)$ is a division algebra based on choosing linearly independent $ca^{-1}=kk'$ and $b^{-1}c=k'$.

  \begin{theorem}\label{thm:div1}
  Suppose that $ca^{-1}$ and $b^{-1}c $ are two linearly independent elements. Suppose one of the following holds, then  ${\rm Men}(K/F, a,b,c)$ is a division algebra over $F$.
  \\ (i)  $ca^{-1}\in F^\times$, but $c a^{-1} \not\in N(K^\times)$, and $b^{-1}c \in K\setminus F $.
  \\ (ii)  $b^{-1}c\in F^\times$,   but $b^{-1}c\not\in N_{K/F}(K^\times)$, and $ca^{-1} \in K\setminus F$.
  \\ (iii)  $b^{-1}c \in K\setminus F$ and $ca^{-1}\in K\setminus F$,
   $kk'^2=b^{-1}c^2 a^{-1}\in F$ and $b^{-1}c^2 a^{-1} \not\in N(K^\times)$.
\\ (iv)  $b^{-1}c \in K\setminus F$ and $ca^{-1}\in K\setminus F $,
 $kk'^2=b^{-1}c^2 a^{-1}\not \in F$ and $b^{-1}c\not \in  N(K^\times)$.
  \end{theorem}

  \begin{proof}
 We assume that
 $$\det (M(y_0, y_1, y_2))
 =N(y_0) + b^{-1}c^2 a^{-1}N_{K/F}(y_2)+ ca^{-1}(N_{K/F}(y_1)-T_{K/F}(y_0\sigma(y_1)\sigma^2(y_2)))=0$$
  and want to conclude that this implies $y_0=y_1=y_2=0$.
  \\ $(i)$  By our assumptions, we have $N_{K/F}(y_0) +  ca^{-1}(N_{K/F}(y_1)-T_{K/F}(y_0\sigma(y_1)\sigma^2(y_2)))=0$ and
  $b^{-1}c^2 a^{-1}N_{K/F}(y_2)=0$, the later implying $N_{K/F}(y_2)=0$ thus $y_2=0$. This means
  $N_{K/F}(y_0) +  ca^{-1}N_{K/F}(y_1)=0$, therefore $ca^{-1}=N_{K/F}(-\frac{y_0}{y_1})\in N_{K/F}(K^\times)$ if $y_1\not=0$, contradicting our assumption. Thus
  $y_1=0$ which means also $y_0=0$.
\\ $(ii)$ By our assumptions, we immediately obtain $N_{K/F}(y_0)=0$, thus $y_0=0$, and so $ b^{-1}c N_{K/F}(y_2)+ N_{K/F}(y_1)=0$.
If also $y_2=0$ then this means $y_1=0$ and we are done, so assume $y_2\not=0$. This yields $ b^{-1}c =N(-\frac{y_1}{y_2})$.
Since we assumed $b^{-1}c\not\in N_{K/F}(K^\times)$ this cannot happen.
\\ $(iii) $ and $(iv)$: We distinguish two subcases.
 If $b^{-1}c^2 a^{-1}\in F$ then we have $N_{K/F}(y_0)+  b^{-1}c^2 a^{-1}N_{K/F}(y_2)=0$ and $N_{K/F}(y_1)-T_{K/F}(y_0\sigma(y_1)\sigma^2(y_2))=0$.
If $y_2=0$ this implies $y_0=y_1=0$. So assume that $y_2\not=0$, then $\alpha \beta =N_{K/F}(-\frac{y_0}{y_2})\in N_{K/F}(K^\times)$, contradicting our assumption.
\\ If $b^{-1}c^2 a^{-1}\not \in F$ then this means $y_0=0$ and thus
$ca^{-1} ( b^{-1}cN_{K/F}(y_2)+ N_{K/F}(y_1))=0$. It follows that $b^{-1}cN_{K/F}(y_2)+ N_{K/F}(y_1)=0$, so that $b^{-1}c\in  N_{K/F}(K^\times)$, a contradiction.
\end{proof}

Since the norm map is surjective for finite fields, the above criteria are not strong enough to check when ${\rm Men}(K/F, a,b,c)$ is a semifield.
A stronger result is obtained when we argue a bit differently:

 \begin{theorem}\label{thm:div2}
  Suppose that $ca^{-1}$ and $b^{-1}c $ are two linearly independent elements. Suppose one of the following holds, then ${\rm Men}(K/F, a,b,c)$ is a division algebra over $F$.
  \\ (i)  $ca^{-1}\in F^\times$, $b^{-1}c \in K\setminus F $ and
   $N_{K/F}(a^{-2} bc)\not\in F^{ \times 3}$.
  \\ (ii)  $b^{-1}c\in F^\times$, $ca^{-1} \in K\setminus F$ and $ N_{K/F}(ab^{-2}c)\not\in N_{K/F}(K^\times)^3$.
  \\ (iii)  $b^{-1}c \in K\setminus F$ and $ca^{-1} \in K\setminus F$,
  $b^{-1}c^2 a^{-1}\not \in F$
  and $ N_{K/F}(ab^{-2}c)\not\in N(K^\times)^3$.
  \end{theorem}

  When $kk'^2=b^{-1}c^2 a^{-1}\in F^\times$ we cannot seem to derive a stronger condition for ${\rm Men}(K/F, a,b,c)$ to be a division algebra over $F$.

\begin{proof}
Suppose that $\det (M(y_0, y_1, y_2))
 =N_{K/F}(y_0) + b^{-1}c^2 a^{-1}N_{K/F}(y_2)+ ca^{-1}(N_{K/F}(y_1)-
 $\\
 $T_{K/F}(y_0\sigma(y_1)\sigma^2(y_2)))=0.$
 \\
$(i)$  By our assumptions  $N_{K/F}(y_2)=0$,
  therefore we have $y_2=0 $. Now
 $$x\cdot y = (x_0y_0  + ca^{-1}x_2 \sigma^2(y_1))+ (x_0y_1 + x_1 \sigma(y_0) )z
+ ( ba^{-1}x_1 \sigma(y_1) + x_2 \sigma^2(y_0))z^2$$
implies $x_0y_0  =- ca^{-1}x_2 \sigma^2(y_1)=0,$ $x_0y_1 =- x_1 \sigma(y_0)$ and $ba^{-1}x_1 \sigma(y_1) =- x_2 \sigma^2(y_0)$.
\\
If also $y_1=0$ then
$$x\cdot y = x_0y_0 + x_1 \sigma(y_0)z  +  x_2 \sigma^2(y_0)z^2=0$$
implies that $x_0y_0=0,$ $ x_1 \sigma(y_0)=0$ and $ x_2 \sigma^2(y_0)=0$. If $y_0=0$ then $y=0$ and we are done, so assume $y_0\not=0$ then $x_0=x_1=x_2=0$ and we are done as well.\\
So let us assume that $y_1\not=0$.\\
If $x_0=0$ then
$$
x\cdot y =  ca^{-1}x_2 \sigma^2(y_1)
+  x_1 \sigma(y_0) z
+ ( ba^{-1}x_1 \sigma(y_1) + x_2 \sigma^2(y_0))z^2
$$
implies $x_2=0$, hence $ ba^{-1}x_1 \sigma(y_1)=0$, so $x_2=0$, thus $x=0$ and we are done.
\\
So assume $x_0\not=0$. Then applying $N_{K/F}$ on both sides of the equations yields
 $$N_{K/F}(x_0)N_{K/F}(y_0 ) =- N_{K/F}(ca^{-1})N_{K/F}(x_2 )N_{K/F}(y_1),$$
  $$N_{K/F}(x_0)N(y_1 )=- N_{K/F}(x_1)N(y_0)$$
   and
   $$N_{K/F}(ba^{-1})N_{K/F}(x_1 )N_{K/F}(y_1) =- N_{K/F}(x_2 )N_{K/F}(y_0).$$
Thus
$- \frac{1}{N_{K/F}(ca^{-1})}\frac{1}{N_{K/F}(y_1)}N_{K/F}(x_0)N_{K/F}(y_0 ) =N_{K/F}(x_2 ),$ and plugging this into the third equation yields
$$N_{K/F}(ba^{-1})N_{K/F}(x_1 )N_{K/F}(y_1) =  \frac{1}{N_{K/F}(ca^{-1})}\frac{1}{N_{K/F}(y_1)}N_{K/F}(x_0)N_{K/F}(y_0 ) N_{K/F}(y_0)$$
i.e.
$$N_{K/F}(ba^{-1})N_{K/F}(x_1 ) =  \frac{1}{N_{K/F}(ca^{-1})}\frac{1}{N_{K/F}(y_1)^2}N_{K/F}(x_0) N_{K/F}(y_0)^2.$$
Assume $y_0=0$. Then by the second equation (since we are assuming $x_0\not=0$), also $y_1=0$ and we are done.
\\ So $y_0\not=0$. By the second equation, we have $N_{K/F}(x_1)= -N_{K/F}(x_0)N_{K/F}(y_1)\frac{1}{N_{K/F}(y_0)}$, so this in turn yields
$$N_{K/F}(ba^{-1})N(x_0)N_{K/F}(y_1)\frac{1}{N_{K/F}(y_0)} =- \frac{1}{N_{K/F}(ca^{-1})}\frac{1}{N_{K/F}(y_1)^2}N_{K/F}(x_0) N_{K/F}(y_0)^2.$$
Therefore, since $x_0\not=0$, we get
$$N_{K/F}(ba^{-1})N_{K/F}(y_1)\frac{1}{N_{K/F}(y_0)} =- \frac{1}{N_{K/F}(ca^{-1})}\frac{1}{N_{K/F}(y_1)^2} N_{K/F}(y_0)^2$$
and hence
$N_{K/F}(bca^{-2}) = N_{K/F}(-\frac{y_0}{y_1})^3.$
This implies $y_1\not=0$.
If $N_{K/F}(a^{-2} bc)\not\in N_{K/F}(K^\times)^3$ then this is a contradiction.
\\ $(ii) $ By our assumptions, we immediately obtain $N_{K/F}(y_0)=0$, thus $y_0=0$. Hence
\begin{align*}
x\cdot y &= (ca^{-1} x_1 \sigma(y_2) + ca^{-1}x_2 \sigma^2(y_1))\\
&+ (x_0y_1 +  cb^{-1}x_2 \sigma^2(y_2))z  \\
&+ (x_0y_2 + ba^{-1}x_1 \sigma(y_1) )z^2
\end{align*}
and this implies $ca^{-1} (x_1 \sigma(y_2) + x_2 \sigma^2(y_1))=0$, $x_0y_1 +  cb^{-1}x_2 \sigma^2(y_2)=0$ and
$x_0y_2 + ba^{-1}x_1 \sigma(y_1) =0$.
\\ If also $y_1=0$ we get $ca^{-1} x_1 \sigma(y_2) =0$, $cb^{-1}x_2 \sigma^2(y_2)=0$ and
$x_0y_2=0$, so either $y_2=0$ as well and we are done, or $y_2\not=0$ but then $x_0=x_1=x_2=0$ and we are done as well.
So assume that $y_1\not=0$, and apply $N_{K/F}$ to both sides of the above equations to obtain
$N_{K/F}(x_1)N_{K/F}(y_2)=- N_{K/F}(x_2)N_{K/F}(y_1)$
$ N_{K/F}(x_0)N_{K/F}(y_1) =- N_{K/F}(cb^{-1})N_{K/F}(x_2 )N_{K/F}(y_2)$, and
$N_{K/F}(x_0)N_{K/F}(y_2) =- N_{K/F}(ba^{-1})N_{K/F}(x_1)N_{K/F}( y_1)$. If also $y_2=0$ we are done, so let $y_2\not=0$ as well. Then the last two equations imply that
$x_0=0$ means $x=0$ so we may assume $x_0\not=0$. They also imply $x_1=0$ means $x=0$ so we may assume $x_1\not=0$.
They also imply $x_2=0$ means $x=0$ so we may assume $x_2\not=0$. We obtain
$N_{K/F}(x_1)=- N_{K/F}(x_2)N_{K/F}(y_1)N_{K/F}(y_2^{-1})$
and plugging this into the third equation yields
$$N_{K/F}(x_0)N_{K/F}(y_2)^2 = N_{K/F}(ba^{-1})N_{K/F}(x_2)N_{K/F}( y_1)^2.$$
Plugging in $ N_{K/F}(x_0) =- N_{K/F}(cb^{-1})N_{K/F}(x_2 )N_{K/F}(y_2)N_{K/F}(y_1)^{-1}$ we get
$$- N_{K/F}(cb^{-1})N_{K/F}(x_2 )N_{K/F}(y_2)N_{K/F}(y_1)^{-1}N_{K/F}(y_2)^2 = N(ba^{-1})N_{K/F}(x_2)N_{K/F}( y_1)^2$$
which simplifies to
$ N_{K/F}(cb^{-1}) = N_{K/F}(ba^{-1})N_{K/F}( -\frac{y_1}{y_2})^3.$
This cannot happen by our assumption.
\\ $(iii)$  If $b^{-1}c^2 a^{-1}\not \in F$ then $y_0=0$ and thus the multiplication becomes
\begin{align*}
x\cdot y &= (ca^{-1} x_1 \sigma(y_2) + ca^{-1}x_2 \sigma^2(y_1))\\
&+ (x_0y_1 +  cb^{-1}x_2 \sigma^2(y_2))z  \\
&+ (x_0y_2 + ba^{-1}x_1 \sigma(y_1) )z^2
\end{align*}
as in $(ii)$. The same argument as in $(ii)$ now shows the assertion.
\end{proof}

\section{Menichetti and crossed product  algebras of degree four} \label{sec:4}

\subsection{Cyclic Field Extensions of Degree 4} \label{degree 4 cyclic}

Let $K/F$ be a cyclic field extension  of degree $4$ with $G={\rm Gal}(K/F)=\langle\sigma\rangle$, $\tau_i=\sigma^i$, norm $N_{K/F}$ and trace $T_{K/F}$. Choose $a,b,c,d \in K^\times$ and let $A={\rm Men}(K/F, a,b,c,d)$.
Then  ${\rm Men}(K/F, a,b,c) $ is a nonassociative $G$-crossed product algebra of degree four.

We already know from Section \ref{sec:results} that
$$\det (M(x_0, x_1, x_2, x_3))=\sum_i a^{n_{i1}}b^{n_{i2}}c^{n_{i3}}d^{n_{i4}} f_i(x_0, x_1, x_2, x_3),$$
where the $n_{ij}$ are integers and the $f_i$ are functions in $x_0, x_1, x_2, x_3$ which take values in $F$,
but we will now explicitly compute the determinant.
Let  $F_0={\rm Fix}(\sigma^2)$ be the intermediate field of $K$ and $F$ fixed by $\sigma^2$. Then
\begin{align} \label{ExplicitDet}
{\rm det}(M(x_0, x_1, x_2, x_3)) & = N_{K/F}(x_0) - \frac{bcd}{a^3}N_{K/F}(x_1) + \frac{c^2d^2}{a^2b^2}N_{K/F}(x_2) \notag \\
& - \frac{d^3}{abc}N_{K/F}(x_3) + \frac{cd}{a^2}T_{K/F}(x_0 \sigma(x_1)\sigma^2(x_1) \sigma^3(x_2)) \notag \\
& - \frac{d}{a}T_{K/F}(x_0 \sigma(x_0) \sigma^2(x_1) \sigma^3(x_3))\notag \\
& + \frac{d^2}{ab}T_{K/F}(x_0 \sigma(x_2)\sigma^2(x_3) \sigma^3(x_3)) \\
& + \frac{cd^2}{a^2b}T_{K/F}(x_1 \sigma(x_2) \sigma^2(x_2) \sigma^3(x_3)) \notag \\
& - \frac{cd}{ab}T_{F_0/F}(x_0\sigma^2(x_0) \sigma(x_2)  \sigma^3(x_2))\notag \\
& + \frac{d^2}{a^2}T_{F_0/F}(x_1\sigma^2(x_1)\sigma(x_3)\sigma^3(x_3)) \notag.
\end{align}
Let $k,k'\in K^\times$. Then ${\rm Men}(K/F,1,k,k,kk')$ has multiplication given by
\[ M(x_0, x_1, x_2, x_3)=
\left( \begin{array}{cccc}
x_0 &  kk'\sigma(x_{3}) & kk'\sigma^2(x_{2}) & kk' \sigma^{3}(x_{1})\\
x_1 & \sigma(x_0)       & k' \sigma^2(x_{3})  &  k' \sigma^{3}(x_{2})\\
x_2 & k\sigma(x_{1})    & \sigma^2(x_0)       & k' \sigma^{3}(x_{3})\\
x_3 &  k\sigma(x_{2})   &  k\sigma^2(x_1)      &  \sigma^{3}(x_{0}) \\
\end {array}\right) .
 \]
 Put $k=a^{-1}b$, $k'=b^{-1}d$, then $(K/F,1,k,k,kk')=(K/F,a,b,b,d)$.

 If $k,k'\in K^\times$, $k=\alpha_1+\alpha_2k'$ with $\alpha_i\in F$, $\alpha_1\not=0$, are chosen such that $1,kk', kk'^2, kk'^3$  are linearly independent over $F$, then
${\rm Men}(K/F,1,k,k,kk')$ is a division algebra over $F$ (Theorem \ref{thm:main1}).
 Corollary \ref{cor:M1}  becomes:

\begin{corollary}
 Let $k,k'\in K^\times$.\\
(i) Let $k=\alpha_1+\alpha_2k'$ with $\alpha_i\in F$, $\alpha_1\not=0$, such that
 $\{1,k', k'^2, k'^3\}$ is a basis for $K/F$.
If $k'^4 = \lambda_0 + \lambda_1k' + \lambda_2k'^2 + \lambda_3k'^3$, such that
$\alpha_1^3 +\alpha_1^2\alpha_2 \lambda_3 - \alpha_1\alpha_2^2 \lambda_2 + \alpha_2 \lambda_1 \neq 0$, then $1, kk', kk'^2, kk'^3$ are linearly independent over $F$, and ${\rm Men}(K/F, 1,k,k,kk')$ is a division algebra over $F$.
\\ (ii) Let  $k'=\alpha_1+\alpha_2k\in K^\times$, $\alpha_i\in F$, $\alpha_1\not=0$, such that $1,k'k,k'k^2,k'k^{3}$ are linearly independent over $F$. Then ${\rm Men}(K/F,1,k,k,kk')$ is a division algebra  over $F$.
\end{corollary}

\begin{lemma}\label{4x4Ex1}(Steele \cite{ST})
Pick $d \in K\setminus F$ such that $1, d, d^2, d^3$ are linearly independent over $F$. Then the following are division algebras:
 ${\rm Men}(K/F, d,1,1,1)$,
 ${\rm Men}(K/F, 1,d,1,1)$,
 ${\rm Men}(K/F, 1,1,d,1)$,
 ${\rm Men}(K/F, d,d,1,d)$,
 ${\rm Men}(K/F, d,1,d,d)$, ${\rm Men}(K/F, 1,d,d,d)$.
\end{lemma}

\begin{proof}
   For ${\rm Men}(K/F, d,1,1,1)$, the explicit formula for ${\rm det}(M(x_0, x_1, x_2, x_3))$ from (\ref{ExplicitDet}) is a polynomial of degree 3 in $1/d$ with coefficients in $F$. Since $1, d, d^2, d^3$ are linearly independent over $F$, so are $1, 1/d, 1/d^2, 1/d^3$. Moreover, the only  term  without a factor of $1/d$ is $N_{K/F}(x_0)$ and the only term with a factor of $1/d^3$ is $-N_{K/F}(x_3)$. Hence if ${\rm det}(M(x_0, x_1, x_2, x_3)) = 0$, then $x_0 = x_3 = 0$. Putting this back into (\ref{ExplicitDet}),  the only remaining coefficient of $1/d$ is $-N_{K/F}(x_1)$ and the only remaining coefficient of $1/d^2$ is $N_{K/F}(x_2)$, so these must be zero also, implying that $x_1 = x_2 = 0$.
The other cases are proved similarly.
\end{proof}

 Note that if $c^{-1}d \in K \setminus F$ then ${\rm Men}(K/F, c,c,c,d)^{op}$
 is a nonassociative cyclic algebra,  and if $c^{-1}d \in F$ then $(K/F, c,c,c,d)^{op}$ is an associative cyclic algebra of degree $4$.

\subsection{Menichetti algebras from biquadratic field extensions}

  Let now  $char(F)\not=2$, $K = F(\sqrt{u}, \sqrt{v})$ be a biquadratic field extension with norm $N_{K/F}$ and trace $T_{K/F}$,
${\rm Gal}(K/F) = \{1, \sigma, \tau, \sigma \tau\},$ where
$\sigma(\sqrt{u}) = \sqrt{u},$ $\sigma(\sqrt{v}) = -\sqrt{v},$
$\tau(\sqrt{u}) = -\sqrt{u}$, $\tau(\sqrt{v}) = \sqrt{v},$ and $\tau_1=\sigma$, $\tau_2=\tau$, $\tau_3=\sigma\tau$.
Let $a,b,c,d \in K^\times$, $A={\rm Men}(K/F, a,b,c,d) $ and $F_0 = F(\sqrt{v})={\rm Fix}(\tau)$.
Then
\begin{align} \label{ExplicitDet2}
{\rm det}(M(x_0, x_1, x_2, x_3)) & = N_{K/F}(x_0) - \frac{bcd}{a^3}N_{K/F}(x_1) + \frac{c^2d^2}{a^2b^2}N_{K/F}(x_2) \notag \\
& - \frac{d^3}{abc}N_{K/F}(x_3) + \frac{cd}{a^2}T_{K/F}(x_0 \sigma(x_1)\tau(x_1) \sigma\tau(x_2)) \notag \\
& - \frac{d}{a}T_{K/F}(x_0 \sigma(x_0) \tau(x_1) \sigma\tau(x_3))\notag \\
& + \frac{d^2}{ab}T_{K/F}(x_0 \sigma(x_2)\tau(x_3) \sigma\tau(x_3)) \\
& + \frac{cd^2}{a^2b}T_{K/F}(x_1 \sigma(x_2) \tau(x_2) \sigma\tau(x_3)) \notag \\
& - \frac{cd}{ab}T_{F_0/F}(x_0\tau(x_0) \sigma(x_2)  \sigma\tau(x_2))\notag \\
& + \frac{d^2}{a^2}T_{F_0/F}(x_1\tau(x_1)\sigma(x_3)\sigma\tau(x_3)), \notag .
\end{align}
Let us consider the special case  ${\rm Men}(K/F,1,k,k,kk')$. Here, the multiplication is given by
\[M(x_0, x_1, x_2, x_3)=
\left( \begin{array}{cccc}
x_0 &  kk'\sigma(x_{3}) & kk'\tau(x_{2}) & kk' \sigma \tau(x_{1})\\
x_1 & \sigma(x_0)       & k' \tau(x_{3})  &  k' \sigma \tau(x_{2})\\
x_2 & k\sigma(x_{1})    & \tau(x_0)       & k' \sigma \tau(x_{3})\\
x_3 &  k\sigma(x_{2})   &  k\tau(x_1)      &  \sigma \tau(x_{0}) \\
\end {array}\right)
 \]
 and if $k,k'\in K^\times$ are chosen such that $1,kk', kk'^2, kk'^3$  are linearly independent over $F$, then
${\rm Men}(K/F,1,k,k,kk')$ is a division algebra over $F$ (Theorem \ref{thm:main1}).

\begin{corollary}
 Let $k,k'\in K^\times$, such that one of the following holds:
 \\
(i) $k=\alpha_1+\alpha_2k'$ with $\alpha_i\in F$, $\alpha_1\not=0$, such that $1,k', k'^2, k'^3$ or $1,kk', kk'^2, kk'^3$  is a basis for $K/F$.
 E.g., if $k'^4 = \lambda_0 + \lambda_1k' + \lambda_2k'^2 + \lambda_3k'^3$, $\lambda_i\in F$, and $\alpha_1^3 +\alpha_1^2\alpha_2 \lambda_3 - \alpha_1\alpha_2^2 \lambda_2 + \alpha_2 \lambda_1 \neq 0$, then $1, kk', kk'^2, kk'^3$ is a basis for $K/F$.
 \\ (ii)  $k'=\alpha_1+\alpha_2k\in K^\times$ with $\alpha_i\in F$, $\alpha_1\not=0$, such that $1,k'k,k'k^2,k'k^{3}$ are linearly independent over $F$.
 \\ Then ${\rm Men}(K/F,1,k,k,kk')$ is a division algebra  over $F$.
\end{corollary}

\begin{example}\label{4x4Ex2} (Steele \cite{ST})
Pick $d \in K^\times$ such that $1, d, d^2, d^3$ are linearly independent over $F$, then  ${\rm Men}(K/F, d,1,1,1)$,
 ${\rm Men}(K/F, 1,d,1,1)$,
 ${\rm Men}(K/F, 1,1,d,1)$,
 ${\rm Men}(K/F, 1,1,1,d)$,
 ${\rm Men}(K/F, d,d,d,1)$,
 ${\rm Men}(K/F, d,d,1,d)$,
 ${\rm Men}(K/F, d,1,d,d)$,
 ${\rm Men}(K/F, 1,d,d,d)$ are division algebras.
The proof is straightforward, looking at the determinant.
\end{example}

\begin{example}\label{ex:Me}
Let us consider the special case  ${\rm Men}(K/F,1,1,1,c)$ with  multiplication  given by
\[M(x_0, x_1, x_2, x_3)=
\left( \begin{array}{cccc}
x_0 &  c\sigma(x_{3}) & c\tau(x_{2}) & c \sigma \tau(x_{1})\\
x_1 & \sigma(x_0)       & c \tau(x_{3})  &  c \sigma \tau(x_{2})\\
x_2 & \sigma(x_{1})    & \tau(x_0)       & c \sigma \tau(x_{3})\\
x_3 &  \sigma(x_{2})   &  \tau(x_1)      &  \sigma \tau(x_{0}) \\
\end {array}\right).
 \]
 This algebra is not a nonassociative $G$-crossed product algebra and is  a division algebra if $1,c,c^2,c^3$ are linearly independent over $F$ (Theorem \ref{thm:main1}). Let us look at the opposite algebra.  Observe that in the opposite algebra
 $z_2$ and $\sqrt{b}$ anti-commute and that  $z_2\cdot z_2 = d$. Thus the opposite algebra contains the nonassociative quaternion algebra $(F(\sqrt{b})/F, \tau, c)$  as a subalgebra, which is the classical quaternion algebra $(b,c)_F$ when $c\in F^\times$. Therefore
 ${\rm Men}(K/F,1,1,1,c)$ contains the associative algebra $(b,c)_F$ when $c\in F^\times$ and the proper nonassociative algebra $(F(\sqrt{b})/F, \tau, c)^{op}$  as a subalgebra when $c\in K\setminus F$. Both are crossed product algebras.
\end{example}

\subsection{Nonassociative crossed product algebras from biquadratic field extensions}

  Let again $char(F)\not=2$, $K = F(\sqrt{u}, \sqrt{v})$  with norm $N_{K/F}$, trace $T_{K/F}$, and Galois group
$G={\rm Gal}(K/F) = \{1, \sigma, \tau, \sigma \tau\},$ where
$\sigma(\sqrt{u}) = \sqrt{u},$ $\sigma(\sqrt{v}) = -\sqrt{v},$
$\tau(\sqrt{u}) = -\sqrt{u}$, $\tau(\sqrt{v}) = \sqrt{v}$ and $\tau_1=\sigma$, $\tau_2=\tau$, $\tau_3=\sigma\tau$.

\begin{example}\label{ex: cross}
Put $c_{\sigma,\tau}=c$ and $c_{g,h}=1_K$ for all other $g,h\in G$.
Then $A=(K, G,  {\bf c})=\bigoplus_{g \in G} K z_{g}$ is a nonassociative $G$-crossed product algebra
with the multiplication induced via
\[
(a z_g)(b z_h) =  a g(b) c_{g,h}\, z_{g\circ h}\text{ for all }
 a,b \in K,\ g,h \in G.
\]
This means in particular that
$$ z_\sigma z_\sigma  = 1 \text{ and } z_\tau z_\tau  = 1$$
while we would have
$$ z_\sigma z_\sigma  = z_{\tau} \text{ and } z_\tau z_\tau  =  z_{\sigma\circ \tau}$$
 for  ${\rm Men}(K/F,1,1,1,c)$ in Example \ref{ex:Me}.
The algebra $(K, G,  {\bf c})$ is an associative $G$-crossed product algebra if and only if $c\in F^\times$.
\end{example}

\begin{example}
Let
$(K, G,  {\bf c})=\bigoplus_{g \in G} K z_{g}$ be an associative $G$-crossed product algebra  and assume that $z_g z_h=1 z_{gh}$ for all unequal $g,h\in G$, and $z_\sigma z_\sigma =c_{\sigma,\sigma}=d\in F^\times$ and
$z_\tau z_\tau =c_{\tau,\tau}=e\in F^\times$.
 Then $(K, G,  {\bf c})=(a,d)_F\otimes_F (b,e)_F$ is a biquaternion algebra.

If we choose the same $c_{g,h}$, but now choose $z_\tau z_\tau=c_{\tau,\tau}=e\in K\setminus F$, then
$$(K, G,  {\bf c})=(a,d)_F\otimes_F (F(\sqrt{b}),\tau,e)_F$$
 is a nonassociative $G$-product algebra and is the tensor product of the quaternion algebra $D_0=(a,d)_F$ and  the nonassociative
quaternion algebra $(F(\sqrt{b}),\tau,e)_F$. The nonassociative $G$-crossed product  algebra $(K, G,  {\bf c})$ is isomorphic to the Petit algebra $$D[t,\tau]/D[t,\tau](t^2-e)$$ where $D=D_0\otimes_F K$ and $\tau$ is the unique extension of $\tau$ to $D$. In this case, we know that $(K, G,  {\bf c})$ is a division algebra if $D$ is a division algebra and $t^2-e$ is indecomposable in $D[t,\tau]$ \cite{P66}.
\end{example}

\section{Linear Menichetti  codes} \label{sec:codes}

Let $S$ be a unital commutative ring. A  linear code of length $n$ over $S$ is a free submodule of the $S$-module $S^n$.
In order to be consistent with the existing terminology for skew constacyclic and more general skew polycyclic (SPC) codes, we will  consider the opposite algebras of Menichetti algebras in this section. We will introduce the new  family of linear Menichetti codes that can be viewed as left ideals, in particular as principal left ideals, in their ambient algebras which are the opposite algebras of Menichetti algebras. When $G$ is cyclic, Menichetti codes can be seen as nonassociative $G$-crossed product  codes.

We begin by defining the opposite algebras of Menichetti algebras also over commutative unital base rings, slightly generalising the original definition of the algebra, also in terms of automorphism groups that are allowed.

 Let $S/S_0$ be an extension of commutative unital rings.
Suppose there are $n$ different $\tau_i\in {\rm Aut}_{S_0}(S)$ which form a finite abelian subgroup $G$ of ${\rm Aut}_{S_0}(S)$.
 Denote its elements by $\tau_i$ with $\tau_0=id$, so
$${\rm Gal}(K/F)=\{ \tau_0,\dots,\tau_{n-1} \}$$
i.e.  we choose to have $\tau_i$ represent one fixed element in $G$. We allow $G=\{id\}$ here, in which case, however, we will deviate slightly from our previous notation and consider the matrix
 \[ M(x_0,\dots,x_{n-1})=\left [\begin {array}{cccccc}
x_0 &  c_{{n-1},1}x_{n-1} & ... & c_{1,{n-1}} x_{1}\\
x_1 & x_0 & ... &  c_{2,{n-1}} x_{2}\\
x_2 & c_{1,1}x_{1} & x_0 & c_{3,{n-1}}x_{3}\\
...& ...  & ... & ...\\
x_{n-2} &  c_{n-3,1}x_{n-3}  & ...& c_{{n-1},{n-1}}x_{n-1}\\
x_{n-1} &  c_{n-2,1}x_{n-2}  & ...& x_{0}\\
\end {array}\right ]
 \]
in the definition of the algebra multiplication below, i.e., put $\tau_i=id$ for all $i \in \mathbb{Z}_n$ in the matrix $M$) and obtain an
$S$-algebra we will denote by ${\rm Men}_{op}(S,G, k_0,\dots,k_{n-1})$.

\begin{definition}
 Let  $k_i\in S^\times$ for $i\in\{0,\dots,n-1\}$,  put $c_{i,0}=c_{0,j}=1$ and put
$$c_{i,j}=k_0^{-1}k_1^{-1} \cdots    k_{j-1}^{-1}   k_ik_{i+1}\cdots k_{i+j-1} \text{ for all non-zero } i,j\in \mathbb{Z}_n.$$
Let $M(x_0,\dots,x_{n-1})$ be the matrix from Equation (\ref{equ:1}).
 The $S_0$-algebra ${\rm Men}_{op}(S/S_0,G, k_0,\dots,k_{n-1})$ defined on $S^n$ with multiplication
 $$(x_0,\dots,x_{n-1})\cdot  (y_0,\dots,y_{n-1})=(x_0,\dots,x_{n-1})M(y_0,\dots,y_{n-1})^t$$
 is called the  \emph{opposite algebra of a Menichetti algebra}.
\end{definition}

The multiplication in ${\rm Men}_{op}(S/S_0,G, k_0,\dots,k_{n-1})$ can also be described via
$$ (az_i)\cdot(bz_j)=a\,\tau_i(b)(c_{i,j} z_{i+j}),$$
$$z_i\cdot  z_0=z_0\cdot  z_i =z_i \text{ for all } i\in\mathbb{Z}_n,$$
$$z_i\cdot  z_j=c_{i,j}z_{i+j} \text{ for all } i,j\in\mathbb{Z}_n\setminus \{0\}$$
for all $a,b\in K$, where again we read the indices modulo $n$.

The unital algebra ${\rm Men}_{op}(S/S_0,G, k_0,\dots,k_{n-1})$ is a free left $S$-module of rank $n$.

Suppose that $G=\langle \sigma\rangle$ is a non-trivial cyclic group and we define $\tau_i=\sigma^i$ for all $ i\in\mathbb{Z}_n\setminus \{0\}$. Then ${\rm Men}_{op}(S/S_0,G, k_0,\dots,k_{n-1})$ can be seen as a $G$-crossed product algebra (defined as before, but now over rings) with multiplication
$$(az_g)(b z_h)=h(a) b c_{h,g} z_{g\circ h}$$
for all $a, b\in K$, $g,h\in G$.

\begin{lemma}
Suppose that $G=\langle \sigma\rangle$ is a cyclic group and $\tau_i=\sigma^i$. Then  ${\rm Men}_{op}(S/S_0,G, k_0,\dots,k_{n-1})$ is associative if and only if
 $$c_{i,j}c_{i+j,k}=\tau_i(c_{j,k})c_{i,j+k}$$
 for all $i,j,k$.
\end{lemma}

The proof of this ``twisted co-cycle condition'' is a straightforward calculation. It is well-known for $G$-crossed product algebras.

\begin{proof}

If $(z_iz_j)z_k=z_i(z_jz_k) $ in ${\rm Men}_{op}(S/S_0, G, k_0,\dots,k_{n-1})$ then $(c_{i,j}z_{i+j})z_k= c_{i,j} c_{i+j,k} z_{i+j+k}$ equals
$z_i(z_jz_k)=z_i (c_{j,k}z_{j+k})=\tau_{i}(c_{j,k})z_i z_{j+k}=\tau_{i}(c_{j,k}) c_{i,j+k} z_{i+j+k}$, thus $c_{i,j} c_{i+j,k} z_{i+j+k}=\tau_{i}(c_{j,k}) c_{i,j+k} z_{i+j+k}$ which implies  $c_{i,j} c_{i+j,k} =\tau_{i}(c_{j,k}) c_{i,j+k}$ since the $z_i$ form a basis.
\\
Conversely, assume $c_{i,j} c_{i+j,k} =\tau_{i}(c_{j,k}) c_{i,j+k} $ in  ${\rm Men}(S/S_0, G, k_0,\dots,k_{n-1})^{op}$. Then for all $a,b,c\in S$ we have
$$((az_i)( b z_j)) (cz_k)= a\tau_i(b) c_{i,j}c_{i+j,k}\tau_i(\tau_j(c))z_{i+j+k}$$
and
$$(az_i) ((bz_j) (cz_k))= a\tau_i(b)\tau_i(\tau_j(c))\tau_i(c_{j,k})c_{i,j+k}z_{i+j+k}.$$
The first equation becomes
$$((az_i)( b z_j)) (cz_k)= a\tau_i(b) \tau_i(\tau_j(c)) \tau_{i}(c_{j,k}) c_{i,j+k}z_{i+j+k}$$
 by substituting in $c_{i,j} c_{i+j,k} =\tau_{i}(c_{j,k}) c_{i,j+k} $, so both are identical as $\tau_i\circ \tau_j=\tau_j\circ \tau_i$.
\end{proof}

 Define a module isomorphism
 $$\Phi:S^n\longrightarrow {\rm Men}_{op}(S/S_0,G, k_0,\dots,k_{n-1}),\quad (c_0,c_1,\dots,c_{n-1})\mapsto \sum_{i=0}^{n-1}c_iz_i.$$

   For a linear code $C$  of length $n$ over $S$, we denote by $C(z_0,\dots,z_{n-1})$ the set of elements
 $\sum_{i=0}^{n-1}c_i z_i\in {\rm Men}_{op}(S/S_0,G, k_0,\dots,k_{n-1})$ associated to the set of all codewords $(c_0,\dots,c_{n-1})\in C$.

 \begin{definition}
  A linear code $C$  of length $n$ over $S$ is called a  \emph{Menichetti code}
 if the set of elements $C(z_0,\dots,z_{n-1})$ is a left ideal in the algebra
 ${\rm Men}_{op}(S/S_0,G, k_0,\dots,k_{n-1})$. The algebra ${\rm Men}_{op}(S/S_0,G, k_0,\dots,k_{n-1})$ is called the \emph{ambient algebra} of the Menichetti  code.
  \end{definition}

 Thus a Menichetti code
$C$ of length $n$ consist of all vectors $(h_0,\dots,h_{n-1})\in S^n$ obtained from the elements
 $\sum_{i=0}^{n-1}h_iz_i$ in a left  ideal  of ${\rm Men}_{op}(S/S_0,G, k_0,\dots,k_{n-1})$. The symmetric nature of the matrix
  \[ M(x_0,\dots,x_{n-1})=\left [\begin {array}{cccccc}
x_0 &  c_{{n-1},1}\tau_1(x_{n-1}) & ... & c_{1,{n-1}} \tau_{n-1}(x_{1})\\
x_1 & \tau_1(x_0) & ... &  c_{2,{n-1}} \tau_{n-1}(x_{2})\\
x_2 & c_{1,1}\tau_1(x_{1}) & \tau_2(x_0) & c_{3,{n-1}} \tau_{n-1}(x_{3})\\
...& ...  & ... & ...\\
x_{n-2} &  c_{n-3,1}\tau_1(x_{n-3})  & ...& c_{{n-1},{n-1}} \tau_{n-1}(x_{n-1})\\
x_{n-1} &  c_{n-2,1}\tau_1(x_{n-2})  & ...& \tau_{n-1}(x_{0})\\
\end {array}\right ]^t
 \]
 defining the right multiplication of ${\rm Men}_{op}(S/S_0,G, k_0,\dots,k_{n-1})$ implies some interesting symmetric behavior for
 Menichetti codes.

For the construction of Menichetti codes, it is not relevant whether the ambient algebra is associative or not.
However, we notice that there are no non-trivial two-sided ideals in ${\rm Men}_{op}(S/S_0, k_0,\dots,k_{n-1})$ when ${\rm Men}(S/S_0, k_0,\dots,k_{n-1})$ is a division algebra and $S$ a field, and a left ideal in ${\rm Men}_{op}(S/S_0, k_0,\dots,k_{n-1})$ could be two-sided in certain cases.
Therefore an important question is when a Menichetti algebra $A$ has non-trivial left ideals, and  when has only trivial left ideals (i.e., it yields only trivial codes). We know that $K$ lies in the left, middle and right nucleus.

\begin{lemma}\label{le:trivial}
Let $A$ be a nonassociative algebra. If $g\in {\rm Nuc}_r(A)$ then $Ag$ is a principal left ideal in $A$.
If, additionally, $g\in {\rm Nuc}_m(A)\cap {\rm Comm}(A)$ then $Ag$ is a two-sided ideal.
\end{lemma}

\begin{proof}
If $g\in {\rm Nuc}_r(A)$ then $A(Ag)=(AA)g=Ag$. If $g\in {\rm Nuc}_m(A)\cap {\rm Comm}(A)$ then $(Ag)A=A(gA)=A(Ag)=(AA)g=Ag$ which proves the second assertion.
\end{proof}

\begin{proposition}
Let $K/F$ be a cyclic field extension with Galois group generated by $\sigma$, and $A={\rm Men}_{op}(K/F,G,k_0,\dots,k_{n-1})$.
\\
(i) $A$ is associative if and only if $c_{i,j}\in F^\times$ for all $i,j$.
\\ (ii)  We have
$${\rm Nuc}_l(A)=\bigoplus_{\{i\,|\,\sigma^{i} \text{fixes every } c_{j,l}\}}Kz_i,$$
$${\rm Nuc}_m(A)={\rm Nuc}_r(A)=\bigoplus_{l}\{Kz_l\,|\, c_{j,l}\in F\ \ \forall j\}$$
 the  index set $H$ is a subgroup of $\mathbb Z_n$.
\\ $(iii)$ If $H=0$ then there are no codes associated to the ambient algebra $A$.
\end{proposition}

\begin{proof}
$(i)$ Take the associator $[az_i,bz_j,cz_l]$ of $A$ and use Lemma \ref{le:sym} to obtain
$$[az_i,\,bz_j,\,cz_l]  =a\sigma^{i}(b)\sigma^{i+j}(c) c_{i,j+l} (c_{j,l}-\sigma^{i}(c_{j,l}))z_{i+j+l}.$$
This immediately implies the assertions $(i)$, $(ii)$.
\\ $(iii)$ This follows form the above now using Lemma \ref{le:trivial}: we need $H\neq 0$ to get any proper left ideals, that is any codes.
\end{proof}

A straightforward proof shows that for a  principal left ideal $Ag$ of $A$ we may assume w.l.o.g. that $g=\sum_{i=0}^{r}g_iz_i\in A$ is \emph{monic} in the sense that $g_{r}=1$, if $g_r$ is invertible in $S$ (e.g. when $S$ is a field since $g_r\not=0$).

By choosing the $c_{i,j}$ wisely, we can define Menichetti algebras with a large left nucleus, and then  build Menichetti codes associated to principal left ideals.

\begin{remark}
If the Galois group $\langle\sigma\rangle$ of $K/F$ is cyclic of prime degree  $n=p$, then
$H$ is a subgroup of $\{0,\mathbb Z_p\}$. Now $H=\mathbb Z_p$ implies that the Menichetti algebra is associative. For a \emph{proper nonassociative}
Menichetti algebra of prime degree over a finite field $\mathbb F_q$ we have $\rm Nuc_r=K$, and so by Lemma \ref{le:trivial} we have
$Ag\in\{0,A\}$ so no codes we can construct. Moreover, then all three nuclei equal $K$, here $n=4$ is the smallest degree where we can construct Menichetti codes.
\end{remark}

\begin{lemma} Suppose that $G=\langle \sigma\rangle$ is a cyclic group and we define $\tau_i=\sigma^i$.
Fix some  $r\in \{0,\dots,m-1\}$. Then $z_r\in {\rm Nuc}_r(A)$ if and only if $c_{i,j}\,c_{i+j,r} \;=\; \sigma^{i}(c_{j,r})\,c_{i,j+r}$ for all $i,j\in \mathbb{Z}_m$.
\end{lemma}

\begin{proof}
If
$z_r\in {\rm Nuc}_r(A)$ then $(z_i z_j)z_r=z_i(z_j z_r)$. Since
\[
(z_i z_j) z_r
= (c_{i,j} z_{i+j}) z_r
= c_{i,j} c_{i+j,r} z_{i+j+r},
\]
and
\[
z_i (z_j z_r)
= z_i (c_{j,r} z_{j+r})
= \sigma^i(c_{j,r}) ( z_i z_{j+r})
= \sigma^i(c_{j,r}) c_{i,j+r}  z_{i+j+r}.
\]
this implies
\[
c_{i,j}\,c_{i+j,r} = \sigma^{i}(c_{j,r})\,c_{i,j+r}
\]
for all $ i,j\in \mathbb{Z}_n$.

Conversely, suppose that we have
\[
c_{i,j}\,c_{i+j,r} = \sigma^{i}(c_{j,r})\,c_{i,j+r}
\]
 for all $i,j\in \mathbb{Z}_n$.  Let $y=\sum_{i=0}^{n-1} a_i z_i$ and $z=\sum_{j=0}^{n-1} b_j z_j$ with $a_i,b_j\in K$.
Then
\[
(y z) z_r
 = \Big(\sum_{i,j} a_i\,\sigma^i(b_j)\,c_{i,j}\,z_{i+j}\Big) z_r
 = \sum_{i,j\in\{0,\dots, n-1\}} a_i\,\sigma^i(b_j)\,c_{i,j}\,c_{i+j,r}\, z_{i+j+r}.
\]
On the other hand,
\[
y (z z_r)
=(\sum_{i=0}^{n-1} a_i z_i) \Big(\sum_{j=0}^{n-1} (b_j z_j)z_r\Big)
 = \sum_i a_i z_i \Big(\sum_j b_j c_{j,r} z_{j+r}\Big)
 = \sum_{i,j\in\{0,\dots, n-1\}} a_i\,\sigma^i(b_j c_{j,r})\,c_{i,j+r}\, z_{i+j+r}
\]
$$ = \sum_{i,j\in\{0,\dots, n-1\} } a_i\,\sigma^i(b_j)\,\sigma^i(c_{j,r})\,c_{i,j+r}\, z_{i+j+r}.$$
Hence $(y z) z_r = y (z z_r)$ for all $y,z\in A$.  Therefore $z_r\in N_r(A)$.
\end{proof}

\begin{definition}
 Let $d\in S^\times$. A Menichetti code $C$  of length $n$ over $S$ is called a \emph{$(G,d)$-constacyclic code}, if the set of elements $C(z_0,\dots,z_{n-1})$ is a principal left ideal in the algebra
 ${\rm Men}_{op}(S/S_0,G, 1,\dots,1,d)$.
  \end{definition}

 \begin{example}\label{ex:cyclic}
 The ambient  algebra ${\rm Men}_{op}(S/S_0, G, 1,\dots,1,d)$ of a $(G,d)$-constacyclic code $C_G$ of length $n$ has the affiliated matrix
  \[ M(x_0,\dots,x_{n-1})^t=\left [\begin {array}{ccccccc}
x_0 &  d\tau_1(x_{n-1}) & d \tau_2(x_{n-2}) &... & d \tau_{n-1}(x_{1})\\
x_1 & \tau_1(x_0) & d \tau_2(x_{n-1})&... &  d \tau_{n-1}(x_{2})\\
x_2 & \tau_1(x_{1}) & \tau_2(x_0) & ... & d \tau_{n-1}(x_{3})\\
...& ...  & ... & ...\\
x_{n-2} &  \tau_1(x_{n-3})  & ...& ...& d \tau_{n-1}(x_{n-1})\\
x_{n-1} &  \tau_1(x_{n-2}) &    \tau_2(x_{n-3})  & ...& \tau_{n-1}(x_{0})\\
\end {array}\right ]^t
 \]
  Let $(a_0,\dots,a_{n-1})\in  C$. If $C_G$ over $S$ is a
   $(G,d)$-constacyclic code of lenght $n$, then $(a_0,\dots,a_{n-1})\in  C_G$ implies that also
  $$( d \tau_2(a_{n-2}), d \tau_2(a_{n-1}),\tau_2(a_0),\dots,\tau_2(a_{n-3}))\in C_G,$$
   $$( d \tau_3(a_{n-3}),  d \tau_3(a_{n-2}), d \tau_3(a_{n-1}),\tau_3(a_0),\dots,\tau_3(a_{n-4}))\in  C_G,\dots.$$
Compare this with a skew $(\sigma,d)$-constacyclic code $C$ over $S$ of length $n$ with ambient Petit algebra $S[t;\sigma]/S[t;\sigma](t^n-d)=(S/S_0,\sigma, d)$.  Here,
$$(a_0,\dots,a_{n-1})\in C \Rightarrow  ( d \sigma(a_{n-1}),\sigma(a_0),\dots,\sigma(a_{n-2}))\in  C,$$
  $$ ( d \sigma^2(a_{n-2}), \sigma(d) \sigma^2(a_{n-1}),\sigma^2(a_0),\dots,\sigma^2(a_{n-3}))\in  C,$$
  $$ ( d  \sigma^3(a_{n-3}), \sigma(d)\sigma^3(a_{n-2}), \sigma^2(d) \sigma^3(a_{n-1}),\sigma^3(a_0),\dots,\sigma^3(a_{n-4}))\in  C,\dots.$$
  Since  $(G,d)$-constacyclic codes allow for abelian noncyclic automorphism groups in their definition, there now is more variety for code design.

  Suppose that $G=\langle \sigma \rangle$ is a cyclic group. Then a $(G,d)$-constacyclic code of length $n$ is a skew $(\sigma,d)$-constacylic code  with ambient Petit algebra
  $$S[t;\sigma]/S[t;\sigma](t^n-d)=(S/S_0, \sigma, d)={\rm Men}_{op}(S/S_0, G, 1,\dots,1,d)$$
  which is a nonassociative $G$-crossed product algebra (a nonassociative cyclic algebra over $S_0$).
   Thus $(G,d)$-constacyclic codes are generalisations of those $(\sigma,d)$-constacylic codes which are in one-one correspondence with the principal left ideals in a nonassociative crossed product algebra $S[t;\sigma]/S[t;\sigma](t^n-d)$. For associative crossed product algebras this was observed in \cite{GM}.
\end{example}

 \begin{example}
 Let  $G=\{id\}$. Then the ambient  algebra ${\rm Men}_{op}(S, \{id\}, 1,\dots,1,d)$ of a $(G,d)$-constacyclic code of length $n$  has the affiliated matrix
  \[ M(x_0,\dots,x_{n-1})^t=\left [\begin {array}{ccccccc}
x_0 &  d x_{n-1} & d x_{n-2} &... & d x_{1}\\
x_1 & x_0 & d x_{n-1}&... &  d x_{2}\\
x_2 & x_{1} & x_0 & ... & d x_{3}\\
...& ...  & ... & ...\\
x_{n-2} &  x_{n-3}  & ...& ...& d x_{n-1}\\
x_{n-1} &  x_{n-2} &    x_{n-3}  & ...& x_{0}\\
\end {array}\right ]^t
 \]
  Thus an $(\{id\},d)$-constacyclic code of length $n$ is a classical constacylic code  with ambient associative quotient algebra $S[t]/S[t](t^n-d)={\rm Men}_{op}(S, \{id\}, 1,\dots,1,d)$, an algebra over $S$.
\end{example}

We  consider the special case that the left ideal is principal.

\begin{theorem}\label{thm:newTheorem3.2}
Let $A={\rm Men}_{op}(S/S_0,G, k_0,\dots,k_{n-1})$.  We identify $g=\sum_{i=0}^{r}g_iz_i\in {\rm Nuc}_r(A)$
with the vector $(g_0,\dots,g_{r},0,\dots,0)\in S^n$. Let $C$ be the Menichetti code defined by the left principal ideal $Ag$.
\\ (i)  If $(a_0,\dots,a_{n-1})\in C$ then
$$( \tau_i(c_{n-i}) c_{i,n-i}  ,  \tau_i(c_{n-i+1}) c_{i,n-i+1} ,\dots, \tau_i(c_0), \tau_i(c_1) c_{i,1} z_{i+1}, \tau_i(c_2) c_{i,1} z_{i+2},\dots, \tau_i(c_{n-1-i}) c_{i,n-1-i} ) \in  C$$
for all $i\in \{1,\dots,n \}$.
 Here, we have a shift: $ \tau_i(c_0) $ is in the $i$th position, $\tau_i(c_1) c_{i,1} $ in position $i+1$ etc.
\\ (ii)
 The matrix
 $$G= \left(
 \begin{array}{cccccccc}
 g_{0} &g_{1} &        \cdots&          g_{r}& 0 &\cdots &\cdots & 0\\
 0 & \tau_1(g_{0}) & \tau_1(g_{1})c_{1,1} &\cdots & \tau_{1}(g_{r-1})c_{1,r-1} & 0 &\cdots & 0\\
 \vdots &\ddots &\ddots &\ddots & &\ddots & &\vdots\\
 \vdots & &\ddots &\ddots &\ddots & &\ddots &\vdots\\
 0 &\ldots & &0 & \tau_{n-1-r}(g_{0}) & \tau_{n-1-r}(g_{1})c_{n-1-r,1} &\ldots & \tau_{n-1-r}(g_{{n-k}})c_{n-1-r,1} \\
 \end{array}
 \right)
 $$
  generating $C$ represents the right multiplication $R_g$ with $g$ in $A$,
calculated with respect to the basis $1,z,\dots,z_{n-1}$.
\\ (iii)
The $(G,d)$-constacyclic code $C\subset S^n$ corresponding to the principal left ideal
$Ag $ is a free left $S$-module of dimension $k=n-r$ with generator matrix
$$
 G= \left(
 \begin{array}{cccccccc}
 g_{_0} &g_{_1} &\cdots&g_{_{n-k}}& 0 &\cdots &\cdots & 0\\
 0 & \tau_1(g_{_0}) & \tau_1(g_{_1}) &\cdots & \tau_{1}(g_{_{n-k}}) & 0 &\cdots & 0\\
 \vdots &\ddots &\ddots &\ddots & &\ddots & &\vdots\\
 \vdots & &\ddots &\ddots &\ddots & &\ddots &\vdots\\
 0 &\ldots & &0 & \tau_{k-1}(g_{_0}) & \tau_{k-1}(g_{_1}) &\ldots & \tau_{k-1}(g_{_{n-k}})\\
 \end{array}
 \right)
 $$
calculated with respect to the basis $1,z,\dots,z_{n-1}$.
\end{theorem}

\begin{proof}
Let $Ag$ be the left ideal generated by $g=\sum_{i=0}^r g_iz_i$.
\\ (i) Let $c=(c_0,c_1,\dots,c_{n-1}) \in  C$. Then by definition $\sum_{j=0}^{n-1}c_jz_j\in Ag$, and so also $z_i \sum_{j=0}^{n-1}c_jz_j\in Ag$ for all $z_i$, $i\in \{1,\dots,n \}$.
We compute
$$z_i \sum_{j=0}^{n-1}c_jz_j= \sum_{j=0}^{n-1}z_i c_jz_j= \sum_{j=0}^{n-1} \tau_i(c_j) z_iz_j= \sum_{j=0}^{n-1} \tau_i(c_j) c_{i,j} z_{i+j},$$
 hence
also
 $$( \tau_i(c_{n-i}) c_{i,n-i}  ,  \tau_i(c_{n-i+1}) c_{i,n-i+1} ,\dots, \tau_i(c_0), \tau_i(c_1) c_{i,1} z_{i+1}, \tau_i(c_2) c_{i,1} z_{i+2},\dots, \tau_i(c_{n-1-i}) c_{i,n-1-i} ) \in  C.$$
 Here, $ \tau_i(c_0) $ is in the $i$th position, $\tau_i(c_1) c_{i,1} $ in position $i+1$ etc.
 \\
E.g., $(  \tau_1(c_{n-1}) c_{1,n-1} , \tau_1(c_0) c_{1,0} ,\dots,  \tau_1(c_{n-2}) c_{1,n-1} ) \in  C$.
\\ (ii)
Since $Ag$ is a left ideal, $z_ig\in Ag$ for all $i\in \{1,\dots,n-1\}.$
It is easy to see that $g,z_1g,z_2g,\dots, z_{n-r-1}g$ are linearly independent over $S$, thus form a basis. The rows of the matrix generating $ C$ are thus given by the vectors corresponding to $g,z_1g,z_2g,\dots, z_{n-r-1}g$. We obtain
a generator matrix $ G $ of $ C$ as
 \begin{equation}\label{eq10}
 G= \left(
 \begin{array}{cccccccc}
 g_{_0} &g_{_1} &\cdots&g_{_{n-k}}& 0 &\cdots &\cdots & 0\\
 0 & \tau_1(g_{_0}) & \tau_1(g_{_1}) &\cdots & \tau_{1}(g_{_{n-k}}) & 0 &\cdots & 0\\
 \vdots &\ddots &\ddots &\ddots & &\ddots & &\vdots\\
 \vdots & &\ddots &\ddots &\ddots & &\ddots &\vdots\\
 0 &\ldots & &0 & \tau_{k-1}(g_{_0}) & \tau_{k-1}(g_{_1}) &\ldots & \tau_{k-1}(g_{_{n-k}})\\
 \end{array}
 \right)
 \end{equation}
 where $k=n-r $.
 Moreover, we just showed that the $(G,d)$-constacyclic code $ C\subset S^n$ corresponding to the principal left ideal
$Ag $ is a free left $S$-module of dimension $n-r$.
\\ (iii) follows as special case from (ii).
\end{proof}

We close with an easy example, which gives a taste of the theory.

\begin{example}\label{ex:one}
Let $F=\mathbb F_2$, $ K=\mathbb F_{16}=\mathbb F_2[\alpha]$ a degree 4 field extension with $\alpha^{4}=\alpha+1$ and
  $\sigma(x)=x^{2}$, with norm  $N_{K/\mathbb F_4}(x)=x^{5}$. Then $\mu_5=\{1,\alpha^{3},\alpha^{6},\alpha^{9},\alpha^{12}\}$ are the fith roots of unity. Let $A={\rm Mon}_{op}(K/F,G,1, \alpha, 1,\alpha)$.
 Then we know that
 $ c_{1,1}=c_{1,3}=c_{3,1}=c_{3,3}=\alpha,$ and $ c_{i,j}=1$ otherwise.
We have
\[
  M(x_0,x_1,x_2,x_3)=
  \left(
 \begin{array}{cccc}
    x_0 & \alpha\,\sigma(x_3) & \sigma^{2}(x_2) & \alpha\,\sigma^{3}(x_1)\\
    x_1 & \sigma(x_0)         & \sigma^{2}(x_3) & \sigma^{3}(x_2)\\[2pt]
    x_2 & \alpha\,\sigma(x_1) & \sigma^{2}(x_0) & \alpha\,\sigma^{3}(x_3)\\
    x_3 & \sigma(x_2)         & \sigma^{2}(x_1) & \sigma^{3}(x_0)
  \end{array}
 \right)
\]
with
$(az_i)(bz_j)=ab^{2^{i}}c_{i,j}z_{i+j}$, written out we have the multiplication table
\[
  \renewcommand{\arraystretch}{1.25}
  \begin{array}{c|ccc}
    \cdot & z_1 & z_2 & z_3\\\hline
    z_1 & \alpha z_2 & z_3 & \alpha\\
    z_2 & z_3 & 1 & z_1\\
    z_3 & \alpha & z_1 & \alpha z_2
  \end{array}
\]
This algebra has dimension 16 over $\mathbb F_2$ and is not even third power associative:
$(z_1z_1)z_1=\alpha z_3$ but $z_1(z_1z_1)=\alpha^{2}z_3$.
We have $ {\rm Nuc}_l(A)=K$ and
$${\rm Nuc}_r(A)= {\rm Nuc}_m(A)=K\oplus Kz_2$$
 is an associative subalgebra with $z_2a=\sigma^{2}(a)z_2$ and
$z_2^{2}=1$, i.e. ${\rm Nuc}_r(A)$ is the cyclic algebra $(\mathbb F_{16}/\mathbb F_4,\sigma^{2},1)\cong M_2(\mathbb F_4)$. Now choose
$g=g_0+g_2z_2\in {\rm Nuc}_r(A)$,
so that $Ag$ is a well-defined left ideal by Lemma \ref{le:trivial}. Because $c_{i,2}=1$ for all $i$, the matrix representing right
multiplication $R_g$ splits into  $2\times 2$  block matrices,  one has determinant $N_{K/\mathbb F_4}(g_0)-N_{K/\mathbb F_4}(g_2)$ and the second its $\sigma$-conjugate. Hence
$Ag$ is a proper  ideal if and only if
$g_0,g_2\neq 0$ and $\rho=g_2/g_0 $ satisfies that $\rho^{5}=1.$
Then
 $$ Ag=\{(u,\ v,\ \rho u,\ \sigma(\rho)v)\,|\, u,v\in K\}\subset K^{4}.
$$
The generator matrix of the associated code $C_\rho$ is
\[
\begin{pmatrix}
      g_0 & 0 & g_2 & 0\\
      0 & \sigma(g_0) & 0 & \sigma(g_2)
    \end{pmatrix}.
\]
For example, we obtain for $g=1+z_2$ the code $C_1=\{(u,v,u,v)\}$ with generator matrix
        $$\begin{pmatrix}1&0&1&0\\0&1&0&1\end{pmatrix}.$$
        For $g=1+\alpha^{3}z_2$ we have $C_{\alpha^{3}}=\{(u,v,\alpha^{3}u,\alpha^{6}v)\}$, with
        generator matrix
        $$\begin{pmatrix}1&0&\alpha^{3}&0\\0&1&0&\alpha^{6}\end{pmatrix}.$$

Each such $C_\rho$ is a Menichetti code of length $4$ over $\mathbb F_{16}$ with parameters
$[4,2,2]$.
\end{example}

Generating examples by hand will be a challenge due to the tedious calculations, but can be done computer assisted easily, as the nonassociativity of the ambient algebra (which is usually not even power-associative) is often irrelevant when we do matrix calculations employing the right multiplication matrix. We emphasize that associative Menichetti algebras will also yield new codes.

\subsection{Outlook}

The above Theorem \ref{thm:newTheorem3.2} means that the set of vectors corresponding to the elements $ \{ g,z_1 g,\ldots,z_{k-1}g\} $ in $A$ forms a basis of $C$ and the dimension of $   C $ is $ k=n-r. $
It remains to see how two generators $g$ and $g'$ of the same principal left ideal in
${\rm Men}_{op}(S/S_0,G, k_0,\dots,k_{n-1})$ are related.

In order to find good Menichetti codes, we will have to systematically determine the ideals and the Hamming preserving automorphisms and isomorphisms between these algebras with emphasis on the princical left ideals. Even when we choose $\sigma=id$ we will obtain new codes which are  generalisation of constacyclic codes with ambient algebra $S[t]/S[t](t^n-a)$.

If $S$ is a finite field then all Galois extensions $K$ of $F$ will be cyclic, hence $G$ will be a cyclic group.
For any proper nonassociative Menichetti algebra $A$ over a finite field, we need that the right nucleus is defined via a proper subgroup $H$ of $\mathbb Z_n$, this implies that a generator $g$ of a proper principal left ideal that lies in the right nucleus will enforce that the Hamming distance is bounded above by $H$ and by the number of nonzero coefficients in $g$ (the \emph{weight} of $g$):
\[
  d_{\min}\leq weight(g)\leq |H|\leq n/2 .
\]

\begin{remark}
We got curious: A search employing Claude came to the conclusion that in Example \ref{ex:one}
the opposite Menichetti algebra has exactly five
proper nonzero left ideals; the $C_\rho$ with $\rho\in\mu_5$. All are minimal, all
have $K$-dimension $2$, none is two-sided, and each is generated as $Ag$ by any one of
its $255$ nonzero elements, so also by elements outside ${\rm Nuc}_r$, e.g. $z_1+z_3$
generates $C_1$. Claude then observed that when we generalize Example \ref{ex:one} and suggests $K=\mathbb F_{q^{4}}$, $F=\mathbb F_q$, $\sigma(x)=x^{q}$ and that we choose
$\theta\in K\setminus\mathbb F_{q^{2}}$, then the opposite Menichetti algebra
$ A={\rm Mon}_{op}(K/F,G,1,\theta,1,\theta)$
has $c_{1,1}=c_{1,3}=c_{3,3}=\theta$ and all other $c_{i,j}=1$; hence
$\rm Nuc_m (A)=\rm Nuc_r (A)=K\oplus Kz_2$, it is proper nonassociative (all this we know from our results). Claude then says that all the
proper left ideals arising from applying  Lemma \ref{le:trivial} are
$C_\rho=\{(u,v,\rho u,\sigma(\rho)v)\}$ for the $q^{2}+1$ elements $\rho$ with
$N_{K/\mathbb F_{q^{2}}}(\rho)=1$. It then the entire construction  over
$K=\mathbb F_{81}$, $F=\mathbb F_3$: same
nuclei, same nonassociativity, and all $10$ ideals $C_\rho$ verified closed under left
multiplication, each $[4,2,2]$ over $\mathbb F_{81}$.
\end{remark}

 The Hamming metric is used to evaluate the efficiency of a code. We thus aim to
classify large families of Menichetti codes up to isometry and equivalence analogous as in \cite{Pum2025} and find those with minimum Hamming distance. For associative $G$-crossed product ambient algebras, the isometries were determined in \cite{GM}, for the associative $G$-crossed product ambient algebras $(K/F,\sigma,d)$ in \cite{NevPum2025, Pum2025}.
We will study
$(G,d)$-constacyclic codes in depth in a future paper \cite{Pum25.2}.

\subsection*{Funding sources} The initial draft of this paper was written  while the author was a visitor at the University of Ottawa. She would like to thank the Department of Mathematics and Statistics for its hospitality, and M Nevins for many fruitful discussions.

\subsection*{Disclosure of potential conflicts of interest} The authors have no relevant financial or non-financial interests to disclose.

\subsection*{Data Availability Statement} No data were generated or used for this research.

%*******************************************************************************************%
%****************************************************************************************%


\begin{thebibliography}{1}

\bibitem{A1} A. A. Albert, \emph{Non-associative algebras: II. New simple algebras.}
Annals of Mathematics, 43 (4) (1942), 708–723.


\bibitem{BHMRV} G.~Blachar, D.~Haile, E.~Matzri, E.~Rein, U.~Vishne, \emph{Semiassociative algebras over a field}.
Journal of Algebra 649 (1) 2024, 35-84.
\verb#https://doi.org/10.1016/j.jalgebra.2024.03.006#

\bibitem{BP18}  C. Brown, S. Pumpl\"un,  \emph{The automorphisms of Petit's algebras.}
 Comm. Algebra    46 (2) (2018), 834-849. %doi:10.1080/00927872.2017.1327598.

\bibitem{BP19}  C. Brown, S. Pumpl\"un, \emph{Nonassociative cyclic extensions of fields and central simple algebras}.
 J. Pure Applied Algebra 223 (6) (2019), 2401--2412.

\bibitem{DW21} J. De La Cruz and W. Willems, "Twisted Group Codes," in IEEE Transactions on Information Theory 67 (8),  5178-5184, Aug. 2021,
%doi: 10.1109/TIT.2021.3089003.

\bibitem{DS} S. T. Dougherty, S. Sahinkaya, B. Yildiz, \emph{Skew $G$-codes.}
Journal of Algebra and Its Applications 22 (02) (2021)
%DOI: 10.1142/S0219498823500561

\bibitem{G} M. Guerreiro,  \emph{Group algebras and coding theory}. São Paulo J. Math. Sci. 10 (2) (2016),  346–371.

\bibitem{GM}  Y. Ginosar, A. Rochas Moreno, \emph{Crossed Products and Coding Theory.}  IEEE Transactions on Information Theory 65, no. 10, 6224-6233, Oct. 2019. % doi: 10.1109/TIT.2019.2923652

\bibitem{J96} N.~Jacobson,
``Finite-dimensional division algebras over fields.'' Springer Verlag,
Berlin-Heidelberg-New York, 1996.



\bibitem{M} G.~Menichetti,  \emph{Sopra una classe di quasicorpi distributivi di ordine finito.}
Atti Accad. Naz. Lincei Rend. Cl. Sci. Fis. Mat. Natur. (8), 59 (5) (1975), 339-348.

\bibitem{NevPum2025}  M. Nevins, S. Pumpl\"un,  \emph{When isometry and equivalence for skew constacyclic codes coincide.}
\verb#arXiv:2508.06695v1# [cs.IT]

\bibitem{OeLun} J. \"Oinert, P. Lundstr\"om,
     \emph{Commutativity and ideals in category crossed products},
   Proc. Est. Acad. Sci.,
  %Proceedings of the Estonian Academy of Sciences
    59 (4)  (2010), 338-346.

\bibitem{P66} J.-C.~Petit, \emph{Sur certains quasi-corps g\'{e}n\'{e}ralisant un type d'anneau-quotient}.
 S\'{e}minaire Dubriel. Alg\`{e}bre et th\'{e}orie des nombres 20 (1966-67), 1-18.


\bibitem{Pum25.2} S. Pumpl\"un,  \emph{Codes from Menichetti algebras}. In preparation.


\bibitem{Pum2017} S. Pumpl\"un, \emph{Finite nonassociative algebras obtained from skew polynomials  and possible applications to $(f,\sigma,\delta)$-codes}.
Advances in Mathematics of Communications  \textbf{11} (3) (2017), 615-634.

\bibitem{Pum2025}  S. Pumpl\"un,\emph{Using nonassociative algebras to classify skew polycyclic codes up to isometry and equivalence}.
To appear in Finite Fields and Applications.
\verb#https://doi.org/10.48550/arXiv.2508.10139# [cs.IT]

  \bibitem{SU25} S. Sahinkaya, D. Ustun, \emph{Skew G-codes over $\mathbb {F}_4[u,v]/\langle u^2, v^2, uv-vu \rangle $ and new binary entanglement-assisted quantum error-correcting codes.} Quantum Information Processing 24 198 (2025)



\bibitem{St} A.~Steele, \emph{Nonassociative cyclic algebras.} Israel J. Math. 200 (2014), 361–387.
https://doi.org/10.1007/s11856-014-0021-7

\bibitem{ST} A.~Steele, \emph{Some new classes of division algebras and potential applications to space-time block coding}. PhD Thesis, University of Nottingham (2013).
    \verb#https://eprints.nottingham.ac.uk/13934/~#

\bibitem{W} W.~C.~Waterhouse, \emph{Nonassociative quarternion algebras.} Algebras Groups Geom.
4 (3), 365–378, 1987.


\end{thebibliography}
\end{document}